\pdfoutput=1
\RequirePackage{ifpdf}
\ifpdf 
\documentclass[pdftex]{sigma}
\else
\documentclass{sigma}
\fi

\usepackage{array}

\usepackage[all]{xy} \xyoption{arc}

\usepackage{tikz}

\newtheorem{thm}{Theorem}[section]
\newtheorem{cor}[thm]{Corollary}
\newtheorem{lem}[thm]{Lemma}

\newtheorem{prob}[thm]{Problem}
 { \theoremstyle{definition}
\newtheorem{dfn}[thm]{Definition}

\newtheorem{ex}[thm]{Example}
\newtheorem{rmk}[thm]{Remark} }

\numberwithin{equation}{section}

\newcommand{\cL}{\mathcal{L}}

\newcommand{\cS}{\mathcal{S}}

\newcommand{\bV}{\mathbf{V}}

\newcommand{\veps}{\varepsilon}

\DeclareMathOperator{\Mod}{\mathbf{Mod}}

\DeclareMathOperator{\uhp}{\mathcal{H}}
\DeclareMathOperator{\rk}{rk}

\DeclareMathOperator{\Tr}{Tr}

\DeclareMathOperator{\Hom}{Hom}

\DeclareMathOperator{\im}{im}

\DeclareMathOperator{\GL}{GL}

\DeclareMathOperator{\SL}{SL}

\DeclareMathOperator{\Res}{Res}

\DeclareMathOperator{\Ind}{Ind}
\DeclareMathOperator{\id}{id}

\DeclareMathOperator{\diag}{diag}

\newcommand*{\df}{\mathrel{\vcenter{\baselineskip0.5ex \lineskiplimit0pt
 \hbox{\scriptsize.}\hbox{\scriptsize.}}} =}

\providecommand{\abs}[1]{\left\lvert#1\right\rvert}

\providecommand{\twomat}[4]{\left(\begin{matrix}#1&#2\\#3&#4\end{matrix}\right)}
\providecommand{\stwomat}[4]{\left(\begin{smallmatrix}#1&#2\\#3&#4\end{smallmatrix}\right)}
\providecommand{\twovec}[2]{\left(\begin{matrix}#1\\#2\end{matrix}\right)}
\providecommand{\pseries}[2]{#1[\![ #2 ]\!]}
\providecommand{\lseries}[2]{#1(\!( #2 )\!)}

\newcommand{\QQ}{\mathbb{Q}}

\newcommand{\CC}{\mathbb{C}}

\newcommand{\ZZ}{\mathbb{Z}}
\newcommand{\PP}{\mathbb{P}}

\newcommand{\NN}{\mathbb{N}}

\DeclareMathOperator{\diagg}{diag}
\DeclareMathOperator{\Vir}{Vir}

\DeclareMathOperator{\ch}{\mathfrak{ch}}

\DeclareMathOperator{\PSL}{PSL}
\DeclareMathOperator{\Der}{Der}
\DeclareMathOperator{\Id}{Id}
\DeclareMathOperator{\codim}{codim}
\DeclareMathOperator{\wt}{wt}

\DeclareMathOperator{\ev}{ev}

\begin{document}
\allowdisplaybreaks

\newcommand{\arXivNumber}{2111.04616}

\renewcommand{\PaperNumber}{085}

\FirstPageHeading

\ShortArticleName{Character Vectors of Strongly Regular Vertex Operator Algebras}

\ArticleName{Character Vectors of Strongly Regular Vertex\\ Operator Algebras}

\Author{Cameron FRANC~$^{\rm a}$ and Geoffrey MASON~$^{\rm b}$}

\AuthorNameForHeading{C.~Franc and G.~Mason}

\Address{$^{\rm a)}$~McMaster University, Canada}
\EmailD{\href{mailto:franc@math.mcmaster.ca}{franc@math.mcmaster.ca}}

\Address{$^{\rm b)}$~UCSC, USA}
\EmailD{\href{mailto:gem@ucsc.edu}{gem@ucsc.edu}}

\ArticleDates{Received December 11, 2021, in final form October 13, 2022; Published online October 29, 2022}

\Abstract{We summarize interactions between vertex operator algebras and number theory through the lens of Zhu theory. The paper begins by recalling basic facts on vertex operator algebras (VOAs) and modular forms, and then explains Zhu's theorem on characters of VOAs in a slightly new form. We then axiomatize the desirable properties of modular forms that have played a role in Zhu's theorem and related classification results of VOAs. After this we summarize known classification results in rank two, emphasizing the geometric theory of vector-valued modular forms as a means for simplifying the discussion. We conclude by summarizing some known examples, and by providing some new examples, in higher ranks. In particular, the paper contains a number of potential character vectors that could plausibly correspond to a VOA, but such that the existence of a corresponding hypothetical VOA is presently unknown.}


\Classification{17B69; 18M20; 11F03}

\section{Introduction}\label{s:intro}
The study of 2-dimensional conformal field theory and the related theory of vertex operator algebras (VOAs), seen from both a physical and mathematical perspective, has from its very beginnings been tightly entwined with the theory of modular forms. In response to a recent mathematical trend~-- which traces back nearly four decades in the physics literature~-- of using results on modular forms to help clarify, and even classify, vertex operator algebras subject to various restrictions, it seems timely to summarize these results for both pracitioners of the art, and for those in related fields that might find inspiration and interest in this subject. Some of the physically motived history of this subject can be traced back to papers such as \cite{Bantay, Bantay2, BG2, BCIR, Cardy, ES2, ES1,GK, MMS}. Below we shall mostly avoid discussion of the important physical aspects of this work, but we thank an annonymous referee for suggesting the inclusion of these references. We refer the reader to Section~\ref{s:voa} for further discussion of definitions and to Section~\ref{SSnotation} for notation.

One of the earliest and to this day most spectacular connections between VOAs and modular forms is of course the construction by Frenkel, Lepowsky and Meurman \cite{FLM} of the monster module denoted by $V^{\natural}$, a holomorphic VOA (i.e., it has only one irreducible module) whose character is the function $j-744$, where $j$ denotes the $j$-invariant of elliptic curves, and whose symmetry group is the Monster simple group. The passage of time only solidifies the status of this construction as the most natural description of the Monster group, notwithstanding the currently open conjecture that the monster module is the unique such VOA. This situation might change if an affirmative answer to Hirzebruch's Preisfrage is found \cite{HirzebruchEtAl}. That is to say, does there exist a $24$-dimensional Monster manifold with elliptic genus equal to $j-744$?

Several years after the appearance of \cite{FLM}, in a groundbreaking work, Zhu showed \cite{Zhu} that the representation theory of a certain class of VOAs essentially generalizes the connection between the $j$-invariant and the monster module. Zhu's work was systematized and broadened in scope in~\cite{DLMModular} (see also the recent paper~\cite{Codogni} of Codogni for a geometric approach), and these trends led to an axiomatic framework for studying the class of VOAs such as $V^{\natural}$ that have only finitely many isomorphism classes of irreducible modules and their \emph{characters}~-- generating series that can be regarded as analogues of the characters of a finite group~-- that are components of vector-valued modular forms. Huang later showed~\cite{Huang} that the representation category of such VOAs has the structure of a modular tensor category \cite{BakalovKirillov}. More recently, Dong--Lin--Ng showed, under some assumptions~\cite{DongLinNg}, that the graded traces appearing in Zhu's theorem have a congruence kernel, a belief which prior to the work of Dong--Lin--Ng had already attained the status of an axiom in the physics literature. The techniques used in \cite{DongLinNg} are, very approximately, a conjunction of the methods of modular tensor categories and the methods established in~\cite{Zhu} and~\cite{DLMModular}.
Very recently, Calegari--Dimitrov--Tang \cite{CDT} established the unbounded denominators conjecture, which dates back to \cite{ASD}, but in a more general vector-valued form suitable for applications in VOA theory and conformal field theory. This theorem states roughly that a vector-valued modular form with integer Fourier coefficients must have components that are classical scalar-valued congruence modular forms. Since characters of VOAs patently have nonnegative integer Fourier coefficients, the spectacular work of Calegari--Dimitrov--Tang makes the congruence nature of characters of VOAs manifestly obvious.\footnote{Andr\'e already observed in the appendix to \cite{Andre} that $p$-curvature techniques as employed in \cite{CDT} could be applied to questions in conformal field theory.} Furthermore it brings clarity to previous work by removing superfluous assumptions.

Having set the stage, the two aims of this paper are to:
\begin{enumerate}\itemsep=0pt
\item[1)] explain in somewhat new and hopefully accessible terms exactly what is proved in~\cite{Zhu} and~\cite{DLMModular},
\item[2)] describe how families of modular forms have played a r\^{o}le in using Zhu's theorem to classify some classes of VOAs.
\end{enumerate}
This program is essentially coextensive with the implementation of what is known as the \textit{modular bootstrap} in physics-speak. It was introduced
in a visionary paper of Mather, Mukhi and Sen~\cite{MMS}.

When discussing classification of VOAs one must always be conscious that the general problem is at least as hopeless as the classification of even positive-definite unimodular lattices. Nevertheless, as we shall explain, there are several interesting cases where a classification has been successful, and these efforts have led to the discovery of new and interesting VOAs. Moreover, while classification of nice lattices is considered intractable in general, greater success has been found in establishing useful \textit{mass formulas}. If such a situation were to pertain also to VOAs, these types of classification results could very well aid in discovering analogous mass formulas. Indeed, it is our considered opinion that establishing a mass formula for suitable classes of VOAs, for example, strongly regular holomorphic VOAs of fixed central charge~$c$, is one of the main open problems that currently exists in VOA theory. Similar comments, less forcibly stated, can already be seen in Schellekens' paper \cite{Schellekens}. We might add, however, that there is currently not even a conjecture in the literature as to what such a mass formula might look like. Furthermore it is presently unknown\footnote{For a nontrivial strongly regular, holomorphic VOA, $c=8k$ is a positive integer divisible by $8$. There is a~unique iso class for $k=1$, two iso classes if $k=2$, whereas finiteness is unknown for any $k\geq 3$.} if there are only \emph{finitely many} strongly regular holomorphic VOAs with a fixed value of $c$! Other than these remarks, though, we alas make no contribution to the problem of describing mass formulas for VOAs in this paper.

A unifying theme throughout this paper is the study of families of monodromy representations, modular forms, and differential equations. To see how these ideas arise and intersect with the study of VOAs, consider the problem, raised and solved by Schellekens (loc.\ cit.) at the level of physical rigor, of classifying strongly regular holomorphic VOAs with $c=24$. The character of such a VOA is a weakly holomorphic modular form for $\SL_2(\ZZ)$ of weight zero with at most a simple pole at the cusp. The space of such forms is spanned by the constant function and the modular $j$-invariant. Moreover, the character of such a holomorphic VOA takes the form $\tfrac 1q + O(1)$ where $q= {\rm e}^{2\pi{\rm i} \tau}$, so that such characters look like $j+m$ for some $m\in \ZZ$. This is a~simple example of a one-parameter family of modular forms. Since for each $m\in \CC$ the form $j+m$ transforms under the trivial representation of $\SL_2(\ZZ)$, so that the monodromy representation is constant along the family, such a~family is termed \emph{isomonodromic}~\cite{Sabbah}. Schellekens showed why one should expect exactly $71$ values of $m$ that correspond to holomorphic VOAs, and, except possibly for the monstrous case $m=-744$ which remains open, it turns out that there is a~\emph{unique} such VOA for each of these values of~$m$. The rigorous proof of this difficult result is due to many mathematicians, too many to cite here. For a good survey see the paper~\cite{LamShimakura} of Lam and Shimakura.

Throughout all of these works the subspace of the VOA of conformal weight $1$ has the structure of a reductive Lie algebra, and the classification of such algebras plays, as it does in much of VOA theory, a major r\^{o}le.\footnote{One of the issues with the case $m=-744$ is that the relevant Lie algebra is then zero-dimensional!} It is a remarkable classification, particularly as any integer $m\geq -744$ yields a plausible choice $j+m$ for a character of a holomorphic VOA, yet most such values of $m$ do \emph{not} correspond to any holomorphic VOAs of central charge $24$!\ This example illustrates the reality that classifying modular forms with nice arithmetic properties~-- the conformal modular forms of Section~\ref{s:conformal}~-- is far from sufficient for classifying characters of VOAs, let alone the VOAs themselves.

In spite of this, there are several recent examples where these two \textit{a priori} very different classification problems have more or less been equivalent. These somewhat more tractable cases have involved VOAs with a small number of irreducible modules~-- typically, at most three~-- and such that the corresponding characters are realized as specializations of families of modular forms that vary in a \emph{purely monodromic} deformation: that is, any small variation of the deformation parameters necessarily changes the underlying monodromy representation. In these purely monodromic cases it has been true that any vector-valued modular form that looks like it \emph{could} plausibly be the character of a VOA \emph{is} the character of a unique VOA. We make these observations precise below, and observe that the purely monodromic deformations are the exception rather than the rule, and thus in general one expects there to be a large gulf between the questions of classifying arithmetically nice modular forms and any corresponding VOAs. Put another way, the difficulties encountered in formulating Schellekens-list type results are typical rather than exceptional.

Let us now briefly describe the contents of each section of the present paper. In Section~\ref{s:voa} we introduce some basic definitions and notations on VOAs and state Zhu's theorem. In Section~\ref{s:vvmf} we describe the basic properties of vector-valued modular forms, recalling in particular the free-module theorem. We state and prove a form (Theorem~\ref{FMT2}) of the free-module theorem that is suited to a discussion of characters of VOAs; this result appeared previously as Theorem~3.3 of~\cite{Gannon} without proof, and for convenience we supply one here. In Section~\ref{s:zhu2} we return to a~discussion of the ideas surrounding Zhu's theorem. We prove a pithy version (Theorem~\ref{thmgradedZ}) that perhaps has independent interest. It emphasizes the r\^{o}le of vector-valued modular forms and the modular derivative $D$. This operator appears only implicitly in Zhu's work, but emphasizing its r\^{o}le clarifies how the relevant differential equations and their monodromy representations arise.\looseness=-1

In Section~\ref{s:conformal} we return to a discussion of modular forms and in particular introduce a notion of \emph{conformal modular form} that incorporates the properties possessed by a modular form that \emph{is} the character of a VOA. In Section~\ref{s:frobenius} we introduce the notion of a \emph{Frobenius family} of modular forms, which is essentially equivalent to the specification of a family of ordinary differential equations on the moduli space of elliptic curves. With these definitions and ideas in place, we can thus explain how most classification results of VOAs have proceeded in four main steps:
\begin{enumerate}\itemsep=0pt
\item[(1)] Impose restrictions on number of simple modules, central charge and conformal weights of the VOAs of interest.
\item[(2)] Write down the most general Fuchsian ODE whose solutions satisfy these conditions.
\item[(3)] Identify all of the conformal specializations of the Frobenius family of modular forms solving the ODE in step (2).
\item[(4)] Identify which of these conformal specializations is in fact realized by a VOA.
\end{enumerate}
While steps (1) and (2) are generally straightforward, both of steps~(3) and~(4) can be quite nontrivial in general.

In Section~\ref{s:rank2} we examine known classification results of VOAs with two simple modules in detail, and explain how all the cases examined thus far have involved purely monodromic Frobenius families, and steps~(3) and~(4) above have been more or less equivalent. Byproducts of this discussion are some strange formulas (Theorem~\ref{t:rank2}) for the dimensions of the graded pieces of the nontrivial module of a VOA with exactly two simple modules (and small~$c$) in terms of values of the $\Gamma$-function. These formulas arise from that fact that the $S$-matrix should transform as a symmetric matrix, and this $S$-matrix is equal up to conjugation to classical monodromy matrices that arise in the study of the Gauss hypergeometric function. Finally, in Section~\ref{s:highrank} we discuss similar computations for VOAs with more than two simple modules, identifying in particular some potential characters of presently unknown VOAs: see for example equation~\eqref{eq:H} on page~\pageref{eq:H} and Table~\ref{rank4examples} on page~\pageref{rank4examples}.

While classification in general is almost surely a hopeless endeavour, it is equally true that there remain many interesting unexamined questions and computations, and in that regard we hope that readers will find these ideas and computations interesting.

\subsection{Notation}\label{SSnotation}
We collect some notation that we use.
\begin{itemize}\itemsep=0pt
\item[--] $\CC$ is the set of complex numbers.
\item[--] $\ZZ$ is the set of rational integers.
\item[--] $\NN$ is the set of positive integers.
\item[--] $\NN_0$ is the set of nonnegative integers.
\item[--] $\Gamma = \SL_2(\ZZ)$ is the homogeneous modular group.
\item[--] $T = \stwomat 1101$, $S =\stwomat 0{-1}1{\hphantom{-}0}$, $R=ST$ are special elements of $\Gamma$.
\item[--] $J= \stwomat {-1}0{\hphantom{-}0}1$.
\item[--] $\uhp = \{x+{\rm i}y\in\CC \mid y > 0\}$ is the complex upper half-plane.
\item[--] $q$ is a formal variable, or $q={\rm e}^{2\pi{\rm i} \tau}\ (\tau\in \uhp)$, depending on context.
\item[--] $B_n$ is the $n^{\rm th}$ Bernoulli number, defined by $\tfrac{z}{e^z-1}=\sum_{n\geq 0} \tfrac{B_n}{n!}z^n$.
\item[--] $G_k(\tau)=-\tfrac{B_k}{k!}- \tfrac{2}{(k-1)!} \sum_{n\geq 0} \sigma_{k-1}(n)q^n$ for $k \geq 2$ is the weight-$k$ Eisenstein series.
\item[--] $E_k = -\frac{k!}{B_k}G_k=1+\cdots $ is the normalized Eisenstein series.
\item[--] $j=q^{-1}+744+196884q+\cdots$ is the modular invariant of elliptic curves.
\item[--] $K=\tfrac{1728}{j}$.
\item[--] $M=\bigoplus_{k\geq 0} M_k=\CC[E_4, E_6]$ is the $2\ZZ$-graded algebra of holomorphic modular forms on $\Gamma$.
\item[--] $M^!=\CC\big[E_4, E_6,\Delta^{-1}\big]$ is the $2\ZZ$-graded algebra of holomorphic modular forms on $\Gamma$ with at worst finite order poles at the cusp.
\item[--] $D_k=\tfrac{1}{2\pi{\rm i}\tau}\tfrac{{\rm d}}{{\rm d}\tau}-\frac{k}{12}E_2=\tfrac{1}{2\pi{\rm i}\tau}\tfrac{{\rm d}}{{\rm d}\tau}+kG_2$ is the $k^{\textrm{th}}$ modular derivative.
\item[--] $D\colon M \rightarrow M$ is the graded operator (derivation) that acts on $M_k$ as $D_k$.
\item[--] $\zeta(z, \tau)$ is the Weierstrass $\zeta$-function.
\item[--] $\wp(z, \tau)= -\partial_z\zeta(z, \tau)$ is the Weierstrass $\wp$-function.
\item[--] $\mathfrak{F}$ is the linear space of holomorphic functions in $\uhp$.
\end{itemize}

\section{Vertex operator algebras}\label{s:voa}
\subsection{Strongly regular VOAs and the character vector}
Throughout this paper we work in the general framework of \emph{strongly regular VOAs}. For a~more thorough discussion of this class of VOAs the reader may consult~\cite{MasonLattice}. Since the focus of our attention is on certain arithmetic and algebraic invariants of such a VOA, the particular set of assumptions we make about our VOAs will play little explicit r\^{o}le in what transpires, although at the end of the day they are not expendable. The main ingredients necessary to define the character vector of a VOA $V$ are that the category of admissible $V$-modules $V$-Mod is \emph{semisimple}, i.e., $V$ is \emph{rational}. One knows~\cite{DLMTwisted} that rationality implies that $V$ has only a finite number of inequivalent simple modules, which we label $M_1, M_2, \dots, M_{d}$. We will also need to know that~$V$ is of \emph{CFT-type}, so that the vacuum vector $\mathbf{1}$ is the unique state of conformal weight~$0$ (up to scalars). For such VOAs it is known~\cite{DongMasonShifted} there are no nonzero states in $V$ of negative weight. (Physicists say that there is a \textit{nondegenerate vacuum}, and that $V$ has \textit{positive-energy}). In any case, the general shape of the decomposition of a CFT-type $V$ into conformal pieces looks like
\begin{gather}\label{confdecomp}
V= \CC\mathbf{1}\oplus V_1\oplus V_2\oplus \cdots.
\end{gather}

Concerning the simple modules, we will usually assume the notation chosen so that $M_1=V$ is the adjoint (or vacuum) module. As before, we take $c$ to be the \emph{central charge} of $V$. Then by the \emph{$q$-character} of $M_j$ we mean the familiar first expression in the next display:
\begin{gather} \label{fjdef}
f_j \df \Tr_{M_j}q^{L(0)-\frac{c}{24}} = q^{h_j-\frac{c}{24}}\sum_{n\geq 0} \dim(M_j)_nq^n.
\end{gather}
That the $q$-character has the general shape given by the second equality above is a consequence~\cite{DLMTwisted} of the simplicity of~$M_j$. The scalar~$h_j$ is fundamental. It is called the \emph{conformal weight} of $M_j$. These mysterious numbers, determined by~$V$, will be a focal point of this paper. For the adjoint module, however, there is no mystery. Because $V$ is of CFT-type~(\ref{confdecomp}) says that $h_1=0$, so that
\begin{gather} \label{fodef}
f_1(q) = \Tr_{V}q^{L(0)-\frac{c}{24}}=q^{-\frac{c}{24}}\sum_{n\geq0} \dim V_nq^n.
\end{gather}
In general the conformal weights could be any complex numbers, but for strongly regular VOAs they are \emph{rational}~\cite{DMRational}.

It behooves us to assemble these individual $q$-characters into a single entity called the \emph{character vector of~$V$}:
\begin{gather} \label{FVdef}
F_V(q)\df \left(\begin{matrix}f_1(q) \\ \vdots \\f_{d}(q)\end{matrix}\right).
\end{gather}

Thus far we have used only a few of the properties of strongly regular VOAs, enough in fact to allow us to define $F_V$. But our main goal is to try to understand the remarkable properties of these character vectors, and for this we need to assume more. We shall discuss Zhu's theorem~\mbox{\cite{DLMModular, Zhu}} more fully in the next subsection, but already in his paper Zhu found it necessary to assume that $V$ is \emph{$C_2$-cofinite}. It is widely believed that this condition is a~\emph{consequence} of rationality, but until this is proved it is convenient to include it in the definition of strong regularity. Strong regularity involves additional properties of $V$ such as the existence of a~nondegenerate invariant bilinear form on~$V$, cf.~\cite{LiSymmetric}. This condition allows us to prove some structural results about $V$, for example $V$ is \textit{simple}~-- an important fact if one is to include the adjoint module as on of the~$M_i$. In some cases we can obtain \emph{characterizations} of certain $V$ by their character vectors.

\subsection{Zhu's theorem}\label{SSZhu}
We continue with a strongly regular VOA $V$. So far the definition of the character vector $F_V$ in equation \eqref{FVdef} is purely formal. An important preliminary result of Zhu \cite{Zhu} reveals the analytic nature of the $q$-characters $f_j(q)$.

This is achieved by showing that each $f_j$ is a solution of a differential equation in the punctured $q$-disk with a regular singularity at $q=0$ and with holomorphic coefficients (actually, the coefficients are holomorphic modular forms). It might be worth pointing out that the $C_2$-condition on $V$ is used here to obtain the existence of the differential equation. Be that as it may, one deduces that the $q$-expansions $f_j$ of equation \eqref{fjdef} are actually the power series expansions obtained by applying the Frobenius method to solve the differential equation in a neighborhood of the regular singularity, and further that this defines a function, that we continue to denote by $f_j$, that is holomorphic throughout $\uhp$ after writing $q = {\rm e}^{2\pi{\rm i} \tau}$ in equation \eqref{fjdef}.

In this way we may consider the $q$-characters as holomorphic functions in $\uhp$ and the character vector as a holomorphic vector-valued function
\begin{gather*}
F_V\colon \ \uhp \longrightarrow \CC^d.
\end{gather*}
In future we shall always regard $F_V(\tau)$ as an $d \times 1$ \emph{column vector} as in~(\ref{FVdef}).

It is convenient to introduce
\begin{gather*}
\ch_V \df \langle f_1(\tau), \dots, f_{d}(\tau)\rangle\subseteq\mathfrak{F},
\end{gather*}
that is, $\ch_V$ is the linear span of the functions $f_j$. The group $\Gamma$ acts on the left of $\uhp$ in the usual way:
\begin{gather*}
\gamma\tau \df \frac{a\tau+b}{c\tau+d}, \qquad \gamma \df \left(\begin{matrix}a & b \\ c & d\end{matrix}\right)\in \Gamma.
\end{gather*}
This left action permits us to make $\mathfrak{F}$ into a right $\Gamma$-module by defining
\begin{gather*}
(f|_0\gamma) (\tau)\df f(\gamma\tau), \qquad f \in\mathfrak{F}.
\end{gather*}

We can now state the main theorem of this section succinctly as follows.
\begin{thm} \label{thmZhu}
Suppose that $V$ is a strongly regular VOA. Then $\ch_V$ is a right $\Gamma$-submodule of~$\mathfrak{F}$.
\end{thm}

 Thus, seemingly out-of-the-blue, we are gifted with a beautiful finite-dimensional $\Gamma$-module which one may~-- among other things~-- study using the methods of representation theory. Still, it has to be emphasized that $\ch_V$ is much more than a mere
 $\Gamma$-module because, unlike standard representation theory, $\ch_V$ has a~\emph{distinguished set of generators} consisting of $q$-characters of the simple $V$-modules and these have an arithmetic structure. This arithmetic structure is the main focus of our interest and the representation of $\Gamma$ on~$\ch_V$ is subservient to it.

 \begin{rmk} The line of argument we have been pursuing leads ineluctably to the study of \emph{vector-valued modular forms}. This is taken up again in Section~\ref{s:vvmf}.
 \end{rmk}

\subsection{The rank of a strongly regular VOA}
 Here we look more closely at the relation
 \[
F_V \longleftrightarrow \ch_V
\]
between the character vector of $V$ and the $\Gamma$-module $\ch_V$.

The first question is, how does $F_V$ manifest the $\Gamma$-module structure of $\ch_V$? In fact nothing more than easy linear algebra is needed to prove the following, once Zhu's theorem is available:
\begin{lem}
 There is a matrix representation $\rho\colon \Gamma \rightarrow \GL_d(\CC)$ such that
\begin{gather*}
\rho(\gamma)F_V(\tau)= F_V(\gamma\tau),
\end{gather*}
and $\rho$ is uniquely determined if $\dim\ch_V=d$.
\end{lem}

The condition $\dim\ch_V=d$ means that the $q$-characters of the simple $V$-modules are linearly independent. If this condition does \emph{not prevail} then $\rho$ in the lemma is typically \emph{not} unique, and the linear space $\ch_V$ loses information about~$V$.

Nevertheless, it is essential to study $\ch_V$ because many typical questions do not involve the linear independence of the~$f_j$. For example: are the functions $f_j(\tau)$ modular functions on a~congruence subgroup of~$\Gamma$ and if so, what are their levels?

The phenomenon that the components of $F_V$ are \emph{not} linearly independent is a common one, and often arises as follows: if $M_j$ is a simple $V$-module, then the dual module $M_j'$ is again a~simple $V$-module and it has the same $q$-character as $M_j$.

The possible discrepancy between $d$ and $\dim\ch_V$, which can sometimes be very large~\cite{MNS2}, veers from the merely annoying to being a serious obstacle. It is convenient to define the \emph{rank} of $V$ as follows:
\begin{gather} \label{dimV}
\rk V \df d.
\end{gather}
That is, $\rk V$ is the number of inequivalent simple $V$-modules. It is a crude but useful measure of the complexity of $V$. Be aware that in the literature the central charge of $V$ has sometimes been called the rank of~$V$, e.g.,~\cite{FrenkelHuangLepowsky}, however this practice is less common nowadays.

\subsection{Holomorphic VOAs}\label{SSholVOAs}
A strongly regular VOA is called \emph{holomorphic} if it has rank $1$ in the sense of definition \eqref{dimV} above. The VOAs in this class have been well-studied both collectively and individually, though much remains to be learned about them. The adjoint module $V$ is the unique simple $V$-module, so the character vector $F_V(\tau)=(f_1(\tau))$ reduces to a single function. By Zhu's theorem, $\ch_V$~furnishes a~$1$-dimensional representation (character) $\alpha\colon \Gamma \rightarrow \CC^{\times}$ via the functional equation
\begin{gather} \label{chidef}
f_1(\gamma\tau)= \alpha(\gamma)f_1(\tau).
\end{gather}
In words, this says that $f_1(\tau)$ is a modular function of weight~$0$ on~$\Gamma$ with a character $\alpha$. Such functions have been well-understood for over a century: every such form equals a power of the Dedekind $\eta$-function times a weakly-holomorphic modular form of level one.

$\CC$ may be regarded as a VOA for which the vertex operator for $v\in \CC$ is multiplication by~$v$. We call this the \textit{trivial} VOA. The trivial VOA is holomorphic with $c=0$.

The character $\alpha$ is subject to some restrictions coming from the VOA structure. We describe these. It is well-known that the multiplicative group of characters of $\Gamma$ is cyclic of order 12, say with a generator $\psi$. So $\alpha$ is a power of $\psi$. We can be more specific:\ it is a convenience that the commutator quotient $\Gamma/\Gamma'$ is generated by the image of $T$.
Thus $\psi$ is uniquely determined by the property that
\begin{gather} \label{psidef}
\psi(T)= {\rm e}^{2\pi{\rm i}/12}.
\end{gather}
 With this notation in place we have
\begin{lem}\label{lemalpha} \label{lempsi}
 $\alpha\in \big\langle \psi^4\big\rangle$, that is, $\alpha$ has order $1$ or $3$.
\end{lem}
\begin{proof}
 This result is well-known \cite{HohnSelbst}. We sketch the proof because it presages some ideas that will recur later.

The statement of the lemma is equivalent to the equality $\alpha(S)=1$, since $T^3=S$ in $\Gamma/\Gamma'$. In order to prove this equality, set $\gamma=S$ and $\tau={\rm i}$ in equation~\eqref{chidef}. The point is that ${\rm i}$ is an elliptic point for $S$, that is, $S{\rm i}={\rm i}$. Then \eqref{chidef} reads $f_1({\rm i})=\alpha(S)f_1({\rm i})$, so it suffices to prove that $f_1({\rm i}) \neq 0$.

Generally, a modular function may well vanish at ${\rm i}$, however this cannot be the case for $f_1(\tau)$. To see this, we use the special nature of the $q$-expansion of $f_1(\tau)$ in equation~\eqref{fodef}. Thus
\begin{gather*}
f_1({\rm i}) = {\rm e}^{-2\pi c/24} \sum_{n\geq 0} (\dim V_n) {\rm e}^{-2\pi n}{\geq}{\rm e}^{{-}2\pi c/24},
\end{gather*}
where the inequality holds because $\dim V_0=1$ and $c\in\QQ$.
\end{proof}

\begin{cor} \label{corc>0} If $V\neq \CC$ is a holomorphic VOA of central charge $c$ then $c\in 8\NN$.
 \end{cor}
 \begin{proof}Since $T$ acts on $\uhp$ as the translation $\tau\mapsto \tau+1$, Lemma \ref{lempsi} yields $f_1(\tau +3)=f_1(\tau)$. In terms of the $q$-expansion \eqref{fodef}, this says
\begin{gather*}
{\rm e}^{-2\pi{\rm i}c(\tau+3)/24}\sum_n (\dim V_n)q^n= q^{-c/24}\sum_n (\dim V_n)q^n\neq 0.
\end{gather*}
It follows that $\tfrac{c}{8}\in\ZZ$.

It remains to prove that $c > 0$. Indeed, we have seen that $f_1(\tau)$ is a nonconstant modular function on $\Gamma$, holomorphic throughout~$\mathcal{H}$, and having a character of order~$1$ or~$3$. As such, it is well-known that $f_1(\tau)$ can only have poles at the cusp, and it has at least one such pole. But the invariance group is either $\Gamma$ itself or the normal subgroup $\Gamma^{(3)}$ of index~$3$, the kernel of~$\psi^4$. Both of these groups have a unique cusp, namely~$\infty$. So $\infty$ is a pole for $f_1(\tau)$ and this means that $-\tfrac{c}{24} < 0$. The corollary is proved.
\end{proof}

\begin{ex}The affine Lie algebra $E_{8,1}$ of level~$1$, which is isomorphic to the $E_8$ lattice theory $V_{E_8}$, is a well-known example of a holomorphic VOA of central charge $c=8$. Indeed, it is the only such VOA~\cite{DMHolomorphic}, and
$\alpha=\psi^4$. Furthermore the usual tensor product $U\otimes V$ of holomorphic VOAs is again holomorphic and central charge is additive over tensor products. As a result, the $p$-fold tensor product $E_{8, 1}^{\otimes p}$ is a holomorphic VOA with $c=8p$. This shows that for a nontrivial holomorphic VOA, $c$ may be any positive integer multiple of~$8$, and that any character $\alpha$ satisfying the conclusions of Lemma~\ref{lemalpha} can occur.
\end{ex}

A better way to organize some of these assertions is as follows: the set $\cS$ of isoclasses of holomorphic VOAs is a commutative graded semigroup with respect to tensor product and with identity $\CC$~-- where the grading semigroup is $8\NN_0$ and comes from the central charge. Thus $\cS_{8p}$ is the set of isoclasses of holomorphic VOAs with $c=8p$.

The question of \emph{classifying} holomorphic VOAs can be an interesting one, depending on what we mean by `classifying'. The size and complexity of $\cS_{8p}$ grows rapidly with $p$. To get a~measure of this, introduce the set $\cL$ of isoclasses of \emph{even, unimodular} lattices regarded as a~semigroup with respect to orthogonal direct sum. If $L\in \cL$ there is a well-known construction of a~VOA~$V_L$ called a lattice VOA \cite{Borcherds, FLM}. Moreover $V_L$ is holomorphic~\cite{DongEven} and the central charge satisfies $c=\rk(L)$. This construction defines the injection $\varphi$ in the diagram below. Note that Corollary~\ref{corc>0} implies that the rank of an even unimodular lattice is divisible by~$8$, a result originally due to Minkowski. In this way we get a commuting diagram of $8\NN_0$-graded semigroups
\begin{gather*}
\xymatrix{ \cL\ar[rr]^{\varphi}\ar[dr]_{rk}& & \cS.\ar[dl]^c\\
&8\NN_0&}
\end{gather*}

Another theorem of Minkowski says that $\cL_{8p}$ is \emph{finite} for each $p$ and it seems very likely that this remains true for $\cS_{8p}$. Indeed, the best way to see the finiteness of $\cL_{8p}$ is by way of a~\emph{mass formula}, see for example~\cite{ConwaySloane}. It is tempting to speculate that there is some sort of mass formula for holomorphic VOAs, but this is unknown. In fact we know very little about $\cS_{8p}$ for $p > 3$. It is known~\cite{DMHolomorphic} that $\varphi$ is bijective at grade~$8$ and $16$, which amounts to $\cS_8=\{V_{E_8}\}$ and $\cS_{16}=\{V_{E_8\perp E_8}, V_{\Gamma_{16}}\}$, where $\Gamma_{16}$ is the \emph{spin lattice} of rank~$16$.

The fundamental classification of Niemeier \cite{ConwaySloane, Niemeier} says that $\abs{\cL_{24}}=24$ whereas Schellekens famously proposed a list of $71$ holomorphic VOAs with $c=24$ in his paper \cite{Schellekens} and conjectured that these are all of the VOAs in~$\cS_{24}$. Giving a rigorous proof of this has turned out to be a major project in its own right though by now it is done~-- except for the uniqueness of the Moonshine module $V^{\natural}$~-- thanks to the work of many. A good survey of this undertaking can be found in the paper of Lam and Shimakura~\cite{LamShimakura}.

It is a well-known consequence of the mass formula that the cardinality of $\cL_{8p}$ is enormous for any $p > 3$~-- for a discussion see, for example~\cite{Serre}. The cardinality of $\cS_{8p}$ is even greater, so it is very large indeed!

\section{Vector-valued modular forms}\label{SD}\label{s:vvmf}

The earlier discussion on character vectors and Zhu's theorem led us naturally to the consideration of functions that transform under the action of the modular group according to a~specific linear representation. We formalize the notion of such functions as \emph{vector-valued modular forms} as follows: given a representation $\rho \colon \Gamma \to \GL_d(\CC)$, not necessarily with finite image or congruence kernel, a~\emph{weakly holomorphic vector-valued modular form} of weight $k \in \ZZ$ for $\rho$ is a~holomorphic function $F \colon \uhp \to \CC^d$ such that the following two conditions are satisfied:
\begin{enumerate}\itemsep=0pt
\item[1)] $F(\gamma \tau) = (c\tau+d)^k\rho(\gamma)F(\tau)$ for all $\gamma = \stwomat abcd \in \Gamma$, and
\item[2)] $F$ is meromorphic at the cusp of $\Gamma$.
\end{enumerate}
\emph{Meromorphy at the cusp} means that for some (and hence any) choice of matrix $L$ such that $\rho(T) = {\rm e}^{2\pi{\rm i} L}$, the function $\tilde F(\tau) = {\rm e}^{-2\pi{\rm i} L\tau}F(\tau)$, which satisfies $\tilde F(\tau+1)=\tilde F(\tau)$ by construction, has a meromorphic $q$-expansion. Let $M^!(\rho)$ denote the space of weakly holomorphic modular forms for $\rho$.

 Theorem~\ref{thmZhu} shows that the character vector $F_V$ associated to a strongly-regular VOA is a weakly holomorphic vector-valued modular form of weight zero. By the solution to the unbounded denominator conjecture in~\cite{CDT}, the representation underlying $F_V$ is known to be unitarizable with a congruence kernel. For unitarizable representations such as this, there do not exist nonzero modular forms of negative weight that are holomorphic at the cusp (cf.~\cite{KnoppMason} or apply the maximum principle for harmonic functions on a compact manifold). Therefore, in the study of $F_V$, cuspidal poles are unavoidable in general. A quantitative version of this fact can be seen more easily, without appealing to unitarity, via a theorem of Dong--Mason~\cite{DMRational} which says that the \emph{effective central charge} of $V$ satisfies $\tilde{c}>0$ for any nontrivial strongly regular VOA~$V$. Bearing in mind the definition of the \emph{effective central charge} of $V$ as $\tilde{c}\df c-24h_{\min}$ where $h_{\min}$ is the minimum of the conformal weights~$h_j$, the inequality says exactly that $F_V$ has at least one cuspidal pole.

For any fixed choice of \emph{exponent matrix} $L$ such that $\rho(T) = {\rm e}^{2\pi{\rm i} L}$, we can define a finite-dimensional space of modular forms $M_k(\rho,L)$ for each integer weight $k$. The space $M_k(\rho,L)$ contains all weakly holomorphic forms whose $q$-expansion has the form
\[
 F(q) = q^{L}f(q),
\]
where $q^L \df {\rm e}^{2\pi{\rm i} L\tau}$ and $f(q) \in \pseries{\CC^d}{q}$. This definition depends on the choice of logarithm $L$ and so it is a crucial piece of the notation. It is explained in \cite{CandeloriFranc} that the geometric role of $L$ is to define an extension of a certain vector bundle to the cusp, so that $M_k(\rho,L)$ is the space of global sections of this bundle. For applications in rational VOA theory, $L$ is diagonal with eigenvalues related to the central charge and conformal weights of the underlying VOA.

\begin{ex}
 Let $\psi$ be the character \eqref{psidef} of $\Gamma$ associated to $\eta^2$. The possible choices of exponents are $\tfrac{1}{12}+n$ for $n \in \ZZ$. In this case
 \[
 \cdots \subseteq M_{k}\big(\psi, \tfrac{13}{12}+n\big) \subseteq M_{k}\big(\psi, \tfrac{1}{12}+n\big) \subseteq \cdots
\]
and $M_k^\dagger(\psi) = \bigcup_{n \ll 0} M_{k}\big(\psi,\tfrac{1}{12}+n\big)$ is the space of weakly holomorphic modular forms of weight~$k$ for $\psi$. Since $\eta^2 = q^{1/12}\prod_{n\geq 1}(1-q^n)^2$, we find that $\eta^2 \in M_1\big(\psi,\tfrac{1}{12}\big)$ but $\eta^2 \not \in M_{1}\big(\psi,\tfrac{13}{12}\big)$. The space $M_1\big(\psi,\tfrac{13}{12}\big)$ contains all weakly holomorphic forms of weight $1$ that vanish to order at least~$\tfrac {13}{12}$ at the cusp, and so it is in fact the zero vector-space.
\end{ex}

For a representation $\rho$ of $\Gamma$ and a choice of exponents $L$, let
\begin{gather*}
M(\rho,L) \df \bigoplus_{k \in \ZZ} M_k(\rho,L)
\end{gather*}
 denote the module of vector-valued modular forms for $(\rho,L)$. The space $M(\rho,L)$ is a graded module over the graded ring $M\df \CC[E_4,E_6]$ of modular forms of level one.
\begin{thm}[free module theorem]\label{FMT}
The $M$-module $M(\rho,L)$ is free of rank $d=\dim \rho$.
\end{thm}
\begin{proof}
This was first proved in \cite{MarksMason} for holomorphic forms. It was then generalized to arbitrary exponents in \cite{CandeloriFranc} using the geometric interpretation of the spaces $M_k(\rho,L)$. See also Theorem~\ref{FMT2} below.
\end{proof}

When $\dim \rho \leq 5$, Marks~\cite{Marks} described the weights of generators of $M(\rho,L)$ in terms of group-theoretic data related to $\rho$. Some of these structural results were expanded upon in~\cite{FM4}, but in general it remains a difficult problem to desribe the free-module structure of $M(\rho,L)$. However, if $\rho$ is unitary and $\rho(-1)=I$, a case of interest in VOA theory, the graded module structure of $M(\rho,L)$ can be determined exactly using the Riemann--Roch theorem, as described in \cite{CandeloriFranc, KnoppMason}. These structural results for $M(\rho,L)$ can be used to show that elements of $M(\rho,L)$ satisfy \emph{modular linear differential equations} of particularly simple forms. In weight zero, the case of interest to us thanks to Zhu's theorem, such modular linear differential equations are equivalent to ordinary differential equations on the modular curve, which is the perspective of~\cite{BantayGannon} and~\cite{Gannon}. In low weights it is easy to use the known structure of $M(\rho,L)$ to determine and solve these equations explicitly: see for example~\cite{FM1, FM2, FM3}, and Sections~\ref{s:rank2},~\ref{s:rank3} and~\ref{s:rank4} below.

If $F\in M_0(\rho,L)$, then the derivative $\frac{1}{2\pi{\rm i}}\frac{{\rm d}}{{\rm d}\tau}F(\tau) = q\frac{{\rm d}}{{\rm d}q}F(q)$ is a modular form in $M_2(\rho,L)$. Continuing to differenitate with respect to~$\tau$ will however not produce modular forms of higher and higher weights. Rather, one gets \emph{quasi-modular forms}~\cite{Zagier}. To recover a modular form, one instead defines a differential operator $D_k \colon M_k(\rho,L) \to M_{k+2}(\rho,L)$ by
\[
 D_k(F) = q\frac{{\rm d}}{{\rm d}q}F -\frac{k}{12}E_2F,
\]
where $E_2$ is as in the notation. We then let $D$ act on $M(\rho,L)$ as a graded derivation. Since $M(\rho,L)$ is free of rank $d = \dim \rho$ over $M$, it follows that every element of $M(\rho,L)$ satisfies a~differential equation in powers of $D$, with coefficients taken in $M$. In particular, the character vectors of strongly regular VOAs satisfy such equations. If $\rho$ is irreducible then the degree of this equation must be $\dim \rho$, and in any case it will be no greater than this.

\begin{rmk} Some care must be taken with the notation when using these operators. For all $k$, we have a perfectly respectable differential operator $D_k\colon \mathfrak{F}\rightarrow\mathfrak{F}$. On the other hand, we have defined $D$ to be the graded differential operator whose restriction to $M_k(\rho, L)$ is $D_k$. Thus, for example, below when we write expressions such as $D^2$, we mean the operator that acts on~$M_k(\rho,L)$ as~$D_{k+2}\circ D_k$.
\end{rmk}

There is an equivalent version of the free-module theorem that is useful for the study of character vectors $F_V$, since it focuses on the space $M^!_0(\rho)$ of modular forms for $\rho$ of weight $0$ with at worst finite order poles at cusps, which is the natural home of $F_V$.
\begin{thm}
 \label{FMT2}
 Suppose that $\rho(-1)$ has $1$ as eigenvalue with multiplicity $e_0$, and $-1$ as eigenvalue with multiplicity $e_1$. Then the space $M^!_k(\rho)$ of weakly-holomorphic forms of weight $k$ is free of rank $e_u$ over $\CC[j]$, where $k \equiv u \pmod{2}$. In particular, if $\rho(-1) = 1$, then $M^!_0(\rho)$ is free of rank $\dim \rho$ over $\CC[j]$.
\end{thm}
\begin{proof}
 This result appears as Theorem~3.3 of~\cite{Gannon} where the proof was sketched. We include the details here, as they illustrate how to pass between the two versions of the free module theorem, which can be useful in applications.

First, since $S^2=-I$ is in the center of $\Gamma$, there is a decomposition $\rho = \rho_0\oplus \rho_1$ where $\rho(-I)$ acts by $1$ and $-1$ on the respective pieces. Since $M^!(\rho) = M^!(\rho_0)\oplus M^!(\rho_1)$, we may without loss of generality assume that $\rho = \rho_0$, so that $\rho(S^2) = 1$ and $k$ is even. Furthermore we may assume without loss of generality that $\rho(T)$ is in Jordan canonical form.\footnote{For applications in nonlogarithmic VOA theory, $\rho(T)$ is in fact \emph{diagonal}, but we do not need to make this hypothesis here.}

Choose any exponents $L$ for $\rho(T)$ and let $F_1,\dots, F_d$ be a free basis for $M(\rho,L) \subseteq M^!(\rho)$ over $\CC[E_4,E_6]$. Let the weights be $k_1,\dots, k_d$. We may assume $k_j \ll k$ for all $j$, at the possible cost of replacing $L$ by $L +nI$ for some integer $n \leq 0$, and with corresponding free basis $\Delta^{n}F_1,\dots, \Delta^nF_d$, where $\Delta$ denotes the usual Ramanujan $\Delta$-function. Therefore, each space $M_{k-k_i}$ is nonzero and contains a basis of the form $b_i, jb_i, \dots, j^{t_i}b_i$. Set $B_i = b_iF_i \in M_k(\rho,L)$ for each $i$, and suppose we have a linear relation
\[\sum_{i=1}^d P_i(j)B_i=0\]
for polynomials $ P_i(j) \in \CC[j]$. Since $j = E_4^3/\Delta$ and $\Delta = \big(E_4^3-E_6^2\big)/1728$, there exists some~$N$ such that for all $i$ we can write $\Delta^NP_i(j) = Q_i(E_4,E_6)$ for $Q_i(E_4,E_6) \in \CC[E_4,E_6]$. If we multiply the displayed equation above by $\Delta^N$, it then follows by the free-module theorem for~$M(\rho,L)$ that $\Delta^NP_i(j) = 0$ for all $i$. Hence $P_i(j) = 0$ for all~$i$. It follows that the $B_i$ are linearly independent over~$\CC[j]$.

Now we show that the $B_i$ span $M_k^!(\rho)$ over $\CC[j]$, which will conclude the proof. Let $F \in M_k^!(\rho)$. Since $F$ has a pole at infinity of finite order, there exists $N \geq 0$ such that $\Delta^NF$ is holomorphic at the cusp, so that $\Delta^NF \in M_{12N+k}(\rho,L)$. Therefore, by the free-module theorem for $M(\rho,L)$ we can write
\[
\Delta^NF = \sum_{i=1}^d Q_iF_i,
\]
where $Q_i \in M_{12N+k-k_i}$. Now $M_{12N+k-k_i}$ contains the basis
\[
 \Delta^Nb_i, \Delta^Njb_i,\dots, \Delta^Nj^{t_i+N}b_i.
\]
But now we see that we can write $\Delta^NF = \sum_{i=1}^d \Delta^NP_i(j)B_i$ for some polynomials $P_i(j) \in \CC[j]$. Cancelling the power of $\Delta^N$ completes the proof.
\end{proof}

\section{Zhu theory revisited}\label{SZhutheory}
\label{s:zhu2}
The purpose of this section is to formulate a revised version of some of the main results in the papers of Zhu \cite{Zhu} and Dong--Li--Mason \cite{DLMModular} in a somewhat more canonical manner that emphasizes the r\^{o}le of vector-valued modular forms and the modular derivative $D$. The main result is stated below as Theorem~\ref{thmgradedZ}. First we assemble the needed pieces in a series of subsections. It will be convenient in this section to work with the renormalized Eisenstein series~$G_k$ as in the notation.

\subsection[The algebras R and S of differential operators]{The algebras $\boldsymbol{R}$ and $\boldsymbol{S}$ of differential operators}\label{SSDO}

Let $R$ denote the \textit{Ore extension} of $M$ by the derivation $D$:
\begin{gather*}
R\df M\langle D \rangle.
\end{gather*}
Multiplication in $R$ is determined by the identity
\begin{gather*}
[D, f] \df (Df-fD)=D(f), \qquad f \in M.
\end{gather*}
The ring $R$ is an entire, associative algebra and elements may be written as (generally noncommutative) polynomials $\sum_i f_iD^i$ for $f_i\in M$. For additional details about such algebras, cf.\ \cite[Chapter 12]{Cohn}. The ring $R$ becomes a $2\ZZ$-graded algebra if we view~$M$ with its usual grading by weight and furnish~$D$ with degree~$2$.

 Given an $\ZZ$-graded entire algebra $A=\oplus_n A_n$, the \emph{Euler operator} of $A$ is the endomorphism $E\colon a\mapsto na$ for $a\in A_n$. The operator $E$ is a derivation of $A$. Now let $E$ be the Euler operator for $R$. We may adjoin $E$ to $R$ to get a second Ore extension
\begin{gather*}
S\df R\langle E\rangle=M\langle D, E\rangle,
\end{gather*}
which is again an associative algebra satisfying
\begin{gather*}
ED-DE= 2D, \qquad Ef-fE= kf, \qquad f \in M_k.
\end{gather*}
We inflict $E$ with degree $0$. Then $S$ is also a $2\ZZ$-graded algebra.

\begin{rmk}
 Both the algebra of derivations $\Der(M)$ and the algebra $S^-$, which means $S$ equipped with the bracket $[a, b]\df ab-ba$ for $a,b\in S$ provide us with Lie algebras which ostensibly have VOA connections. For example, let $W^+$ be the Lie subalgebra of the Witt algebra $W$ spanned by operators $L(n)$ for $n \geq 0$ with bracket $[L(m), L(n)]=(m-n)L_{m+n}$. Regarding $\CC[G_4]E$ as a Lie subalgebra of (either of) the aforementioned Lie algebras, there is an isomorphism $\CC[G_4]E\stackrel{\cong}{\longrightarrow} W^+$ defined by $G_4^mE \mapsto -4L(m)$. There is a similar result concerning~$\CC[G_6]E$.
\end{rmk}

 \subsection{The VOA on the cylinder}
 One of the main geometric ideas in \cite{Zhu} is to formulate the VOA-theoretic version of passage from genus $0$ to genus $1$, that is, sphere to torus (cylinder might be more accurate). This is fundamental for clarifying the r\^{o}le of modular-invariance. We give a brief description, although we will not need too many details. For additional background, cf.\ \cite{DLMModular, MasonTuite, Zhu}, as well as the references cited below.

 For $z\in\CC$ the function $z\mapsto {\rm e}^{z}-1$ maps the punctured Riemann sphere $S$ to an infinite cylinder. Now the VOA axioms are, in a strong sense, related to $S$. The justification for such a statement involves an appeal to the alternative approach to the Jacobi identity in terms of rational functions, see~\cite{FrenkelHuangLepowsky}. This perspective is called "duality" in the physics literature. The formal VOA-theoretic version involves an assignment
 \begin{align*}
(V, Y, \omega, \mathbf{1}) &\mapsto (V, Y[\ ], \widetilde{\omega}, \mathbf{1}).
\end{align*}

The $4$-tuple on the right is called either the VOA \textit{on the cylinder}, or the
\textit{square bracket VOA}, and involves the following ingredients: the underlying Fock space is the same linear space $V$ as the original, and the vacuum elements are also unchanged. The vertex operators are
\begin{gather*}
Y[v, z] \df Y\big({\rm e}^{zL(0)}v, {\rm e}^z-1\big),\qquad
\widetilde{\omega} \df \omega-\tfrac{c}{24}\mathbf{1},
\end{gather*}
where $c$ is the central charge of $V$. It is a remarkable fact that, with this structure, the VOA on the cylinder, which we denote by $V[\ ]$ is, as the name suggests, itself a VOA with Virasoro element $\widetilde{\omega}$. In fact, $V$ and $V[ \ ]$ are isomorphic VOAs, a fact proved by Zhu~\cite{Zhu} in special cases and in general by Lepowsky.

Perhaps the most important aspect (at least for us) is that $V$ has now been equipped with two rather different conformal gradings, coming from the two different Virasoro elements:
\begin{gather} \label{dirsums}
V =\bigoplus_{n\geq0} V_n = \bigoplus_{n\geq0} V_{[n]}.
\end{gather}
Here, if $L[0]$ is the zero mode of $\widetilde{\omega}$, then
\begin{gather*}
V_{[n]} \df\{v\in V\mid L[0]v=nv\}.
\end{gather*}
The two vertex operators $Y(v, z)$ and $Y[v, z]$ are related by rather complicated formulas that are awkward to deal with. For example, we have
\begin{gather*}
L[0] = L(0)+\sum_{n=1}^{\infty}\tfrac{(-1)^n}{n(n+1)}L(n).
\end{gather*}
 We emphasize that equation \eqref{dirsums} is an equality of linear spaces, not $\ZZ$-graded linear spaces. Generally we only have $V_{[n]} \subseteq \oplus_{m\leq n} V_m$.

 \subsection[The graded space M otimes V]{The graded space $\boldsymbol{M \otimes V[\ ]}$} \label{SSMV}

 We continue to let $V$ be a VOA with $V[\ ]$ the isomorphic VOA on the cylinder. In particular, $V$~is equipped with the square bracket conformal grading \eqref{dirsums}. We will be dealing with the tensor product $M\otimes V$ which we consider as a $\ZZ$-graded space equipped with the tensor product grading defined by
\begin{gather*}
(M\otimes V)_n\df \bigoplus_{\ell+k=n} M_{\ell}\otimes V_{[k]}.
\end{gather*}

There are several ways to consider $M\otimes V$ as a vertex ring over $M$ \cite{MasonPierce}, although none of these structures seem to be particularly relevant to the present context. The structure of greatest utility for us is presented in the next result.
\begin{lem}
 \label{lemD}
 $M \otimes V$ is a $\ZZ$-graded left $S$-module where the action for $g\in M$, $f\in M_{\ell}$ and $v\in V_{[k]}$ is given by
 \begin{gather*}
g(f\otimes v) \df gf\otimes v, \\
D(f\otimes v)\df D(f)\otimes v +f\otimes D(v), \\
D(v)\df L[-2]v -\sum_{\ell\geq 2} G_{2\ell} \otimes L[2\ell-2]v,\\
E(f\otimes v)\df (k+\ell)(f\otimes v).
\end{gather*}
\end{lem}
\begin{proof}
 One only has to check the relations for $S$ written down in Section~\ref{SSDO}, which is not hard. One of the main points is that $D$ raises weights by $2$.
\end{proof}

At this point we interject a striking identity of operators, which makes use of the strange definition of the action of $D$ in Lemma \ref{lemD}. We will make no further use of this result.
\begin{lem} \label{thmzeta}
 We have
\begin{align*}
 [D, Y[v, z]]&=\sum_{i\geq0} \tfrac{1}{i!}\zeta^{(i)}(z,\tau) Y[L[i-1]v, z]\\
 &= \zeta(z, \tau)Y[L[-1]v, z]-\wp(z, \tau)Y[L[0]v, z]+\cdots.
\end{align*}
\end{lem}
\begin{proof}
 We start by computing some commutators:
 \begin{gather*}
[D, v[n]](f\otimes u)=D(f\otimes v[n]u)-v[n](D(f)\otimes u+ f\otimes D(u))\\
\hphantom{[D, v[n]](f\otimes u)}{}
=f\otimes D(v[n]u)-v[n]f\otimes D(u)\\
\hphantom{[D, v[n]](f\otimes u)}{} = f\otimes \bigg\{ L[-2]v[n]-\sum_{\ell\geq 2} G_{2\ell} \otimes L[2\ell-2]v[n]\bigg\}(u) \\
\hphantom{[D, v[n]](f\otimes u)=}{}
- f\otimes v[n]\bigg\{L[-2] -\sum_{\ell\geq 2} G_{2\ell} \otimes L[2\ell-2]\bigg\}(u)\\
\hphantom{[D, v[n]](f\otimes u)}{}
=\bigg\{[L[-2], v[n]]- \sum_{\ell\geq 2}G_{2\ell}\otimes[L[2\ell-2], v[n]]\bigg\}(f\otimes u).
\end{gather*}

 We also have
\begin{gather*}
[L[2p], Y[v, z]]=[\tilde{\omega}(2p+1), Y[v, z]]=\sum_{n\in\ZZ} [\widetilde{\omega}(2p+1), v[n]]z^{-n-1}\\
\hphantom{[L[2p], Y[v, z]]}{}
=\sum_n\sum_{i\geq 0} \binom{2p+1}{i}(\widetilde{\omega}[i]v)[2p+1+n-i]z^{-n-1}\\
\hphantom{[L[2p], Y[v, z]]}{}
=\sum_{i\geq 0}\binom{2p+1}{i} z^{2p-i+1}Y[L[i-1]v, z].
\end{gather*}
Therefore
\begin{gather*}
 [D, Y[v, z]]=\sum_{n\in\ZZ} [D, v[n]]z^{-n-1}\\
\hphantom{[D, Y[v, z]]}{}
 = \sum_n \bigg\{ \bigg\{[L[-2], v[n]]- \sum_{\ell\geq 2}G_{2\ell}\otimes[L[2\ell-2], v[n]]\bigg\} \bigg\}z^{-n-1}\\
\hphantom{[D, Y[v, z]]}{}
=[L[-2], Y[v, z]]-\sum_{\ell} G_{2\ell} [L[2\ell-2], Y[v, z]]\\
\hphantom{[D, Y[v, z]]}{}
=\sum_{i\geq 0} (-1)^iz^{-1-i}Y[L[i-1]v, z]-\sum_{i\geq0}\sum_{\ell} G_{2\ell} \binom{2\ell-1}{i} z^{2\ell-1-i}Y[L[i-1]v, z]\\
\hphantom{[D, Y[v, z]]}{}
=\sum_{i\geq0}Y[L[i-1]v, z] \left\{ (-1)^iz^{-1-i}-\sum_{\ell} G_{2\ell} \binom{2\ell-1}{i} z^{2\ell-1-i} \right\}\\
\hphantom{[D, Y[v, z]]}{}
=\left\{ z^{-1}- \sum_{\ell} G_{2\ell}z^{2\ell-1} \right\}Y[L[-1]v, z]\\
\hphantom{[D, Y[v, z]]=}{}+\sum_{i\geq1} (-1)^iz^{-i}\left\{z^{-1}+(-1)^{i-1}\sum_{\ell} G_{2\ell} \binom{2\ell-1}{i} z^{2\ell-1} \right\}Y[L[i-1]v, z]\\
\hphantom{[D, Y[v, z]]}{}
=\zeta(z, \tau)Y[L[-1]v, z]-\wp(z, \tau)Y[L[0]v, z]\\
\hphantom{[D, Y[v, z]]=}{}+
\sum_{i\geq 2}\bigg\{(-1)^iz^{-1-i}-\sum_{\ell} G_{2\ell} \binom{2\ell-1}{i} z^{2\ell-1-i} \bigg\} Y[L[i-1]v, z]\\
\hphantom{[D, Y[v, z]]}{}
=\zeta(z, \tau)Y[L[-1]v, z]-\wp(z, \tau)Y[L[0]v, z] \\
\hphantom{[D, Y[v, z]]=}{}+\sum_{i\geq 2}\bigg\{(-1)^iz^{-i-1}-(i!)^{-1}\partial_z^i\bigg(\sum_{\ell}G_{2\ell}z^{2\ell-1}\bigg) \bigg\} Y[L[i-1]v, z]\\
\hphantom{[D, Y[v, z]]}{}
=\zeta(z, \tau)Y[L[-1]v, z]-\wp(z, \tau)Y[L[0]v, z] \\
\hphantom{[D, Y[v, z]]=}{}
+\sum_{i\geq 2}\bigg\{\tfrac{1}{i!}\partial_z^i\bigg(z^{-1}-\sum_{\ell}G_{2\ell}z^{2\ell-1}\bigg) \bigg\} Y[L[i-1]v, z]\\
\hphantom{[D, Y[v, z]]}{}
=\zeta(z, \tau)Y[L[-1]v, z]-\wp(z, \tau)Y[L[0]v, z] +\sum_{i\geq 2}\tfrac{1}{i!}\partial_z^i\zeta(z, \tau)Y[L[i-1]v, z].
\end{gather*}
The lemma follows.
\end{proof}

One may compare this result with the analog having $E$ in place of $D$. One easily finds
\[
[E, Y[v, z]]=Y[E(v), z]+zY[L[-1]v, z].
\]

 \subsection{1-point functions and the space of (genus 1) conformal blocks}
 We now fix a strongly regular VOA $V$ of central charge $c$ having inequivalent simple modules $V=M_1, \dots, M_d$ and conformal weights $0, h_2, \dots, h_d$. As before, $V[\ ]$ denotes the (isomorphic) VOA on the cylinder.

 For a homogeneous state $v\in V_k$ we define its \emph{zero mode} to be $o(v)\df v(k-1)$. It is well-known, and readily checked from the VOA axioms, that the zero mode has weight $0$ as an operator on $V$ (hence the name), that is,
 \begin{gather*}
o(v) \colon \ V_n\rightarrow V_n,\qquad n \geq 0.
\end{gather*}
One then extends the zero mode notation linearly, so that for an arbitrary state $v\in V$ its zero mode $o(v)$ is the sum of the zero modes of its homogeneous parts.

This formalism extends to $V$-modules. That is to say, the zero mode preserves weight subspaces of each $M_i$:
\begin{gather*}
o(v) \colon \ (M_i)_{n+h_i}\rightarrow (M_i)_{n+h_i}, \qquad n \geq0, \quad i=1, \dots, d.
\end{gather*}
Here we are employing a common abuse of notation inasmuch as $o(v)$ is actually a mode of the vertex operator $Y_{M_i}(v, z)$ corresponding to $M_i$ and should properly be denoted by $o_i(v)$ or something like that. We will desist from this practice, which will cause no confusion.

 Now we can give the first installment of the definition of the $1$-point function defined by $M_i$. It is the function $f_i$ with domain $V$ defined by
\begin{gather*}
f_i(v) \df \Tr_{M_i} o(v)q^{L(0)-\frac{c}{24}}= q^{h_i-\frac{c}{24}}\sum_{n\geq 0} \Tr_{(M_i)_{n+h_i}} o(v)q^n.
\end{gather*}
Notice that $f_i(\mathbf{1})$ is what we called $f_i$ before. And, in fact, Zhu proves that $f_i(v)$ is a holomorphic function in the upper half-plane for each $v\in V$ along lines similar to those that prevailed for~$f_i(\mathbf{1})$. We usually consider $f_i$ as a linear map
\begin{gather*}
f_i\in \Hom_{\CC}(V, \mathfrak{F})
\end{gather*}
 The final incarnation of these $1$-point functions are their $M$-linear extensions
 \begin{gather*}
f_i \in \Hom_{M}(M \otimes V, \mathfrak{F}).
 \end{gather*}

 The function $f_i$ in this form is often called a \textit{conformal block}, especially in the physics literature. The linear span of the $f_i$ is called the space of (genus $1$) \emph{conformal blocks on the torus}, denoted by $B= B_V$.
 One may define conformal blocks of higher genera, but we shall not pursue this. Accordingly, we refer to $B$ simply as the space of conformal blocks. The version of Zhu's modular-invariance theorem proved in \cite{Zhu} may now be stated as follows:
 \begin{thm}[Zhu] \label{thmZ}
 The following hold:
 \begin{enumerate}\itemsep=0pt
\item[$1.$] $\{f_1, \dots, f_n\}$ are linearly independent in $\Hom_{\CC}(V, \mathfrak{F})$.
\item[$2.$] $B$ is a right $\Gamma$-module with respect to the action defined by
 \begin{gather*}
f| \gamma (v, \tau) \df (c\tau+d)^{-k}f(\gamma\tau),\qquad v \in V_{[k]}, \quad \gamma =\twomat abcd \in\Gamma.
 \end{gather*}
\end{enumerate}
\end{thm}

 There is a fundamental characterization of the functions in $B$, as follows.
 \begin{thm} \label{thmB}
 Suppose that $f \in \Hom_{M}(M \otimes V, \mathfrak{F})$. Then $f\in B$ if, and only if, the following conditions hold for all $k$, all $u \in V$, and all $v \in V_{[k]}$:
\begin{gather}
f(u[0]v)=0,\nonumber\\
f\bigg(u[-2]v+\sum_{\ell\geq2} (2\ell-1)G_{2\ell}{\otimes} u[2\ell-2]v \bigg)=0,\nonumber\\
f\circ D(v)=D_k\circ f(v). \label{13}
\end{gather}
\end{thm}
\begin{proof}
See \cite{Zhu} or \cite[Section~5]{DLMModular}.
\end{proof}

 \subsection[The morphism F]{The morphism $\boldsymbol{F}$}

 Our goal in this subsection is to state and prove a more comprehensive version of Theorem~\ref{thmZ}. We use the following additional notation:
 \[\rho\colon \ \Gamma \rightarrow \GL_d(\CC)\]
 is the matrix representation furnished by the conformal block $B_V$ with respect to the basis $(f_1, \dots, f_d)$, and $L$ is the canonical choice of exponents adapted to $V$, namely
\[L\df \diag\big({-}\tfrac{c}{24}, h_2-\tfrac{c}{24}, \dots, h_d-\tfrac{c}{24}\big).\]
The second part of Theorem~\ref{thmZ} may then be restated as follows:
\begin{lem} \label{lem4}
 Let $v \in V_{[k]}$ and set
 \[F(v)\df \left(\begin{matrix}f_1(v) \\ \vdots \\ f_d(v) \end{matrix}\right).\]
Then $F(v) \in M_k(\rho, L)$.
\end{lem}

 There is a useful (though equivalent) variant of Lemma \ref{lem4} as follows:
 \begin{lem}\label{lemeval}
 Let the notation be as in Lemma~{\rm \ref{lem4}}, and let $L_v$ be the linear span of the $1$-point values $f_i(v)$. Then $L_v$ is a right $\Gamma$-module and evaluation at $v$ is a surjective morphism of $\Gamma$-modules
 \begin{gather*}
 \ev_v \colon \ B\longrightarrow L_v,\qquad f_i \mapsto f_i(v).
\end{gather*}
In particular, evaluation at the vacuum $\mathbf{1}$ defines a surjection
\begin{gather*}
 \ev_{\mathbf{1}}\colon \ B\longrightarrow \ch_V,\qquad f_i \mapsto f_i(\mathbf{1}).
\end{gather*}
 \end{lem}

 \begin{ex}
 If $V=V_\Lambda$ is a lattice theory for a positive-definite even lattice $\Lambda$, and if we take $v\df {\rm e}^{\alpha}$ for some nonzero vector $\alpha\in \Lambda$, then the corresponding evaluation map is \emph{trivial}, that is, its image is $0$.
 \end{ex}

Now we have
\begin{thm}
 \label{thmgradedZ}
 The $M$-linear extension of $F$ defines a morphism of $\ZZ$-graded $S$-modules
\begin{gather*}
F\colon \ M\otimes V \longrightarrow \bigoplus_{k\geq 0} M_k(\rho, L).
\end{gather*}
\end{thm}
\begin{proof}
We first consider the $S$-module structures that are involved. We have already introduced the action of $S$ on $M \otimes V$ in Section~\ref{SSMV}. The action of $S$ on $M(\rho, L)$ is discussed in Section~\ref{s:vvmf}: $M$~acts on vector-valued modular forms by pointwise multiplication, and since $M(\rho, L)$ is $\ZZ$-graded we take $E$ to act as its Euler operator. Finally, $D$ acts on $M_k(\rho, L)$ in a componentwise fashion as the $k^{\textrm{th}}$ modular derivative~$D_k$,

 Now we show that $F$ is $S$-linear. Its very definition shows that it is $M$-linear, and by Lemma~\ref{lem4}, $F$~is a graded map, so it commutes with $E$ which acts everywhere as the Euler operator. That $F$ commutes with the action of $D$ is less obvious. Let's see what is involved. For $g \in M_t$, $v\in V_{[k]}$, consider
 \begin{gather*}
(F\circ D)(g \otimes v) =F(D_t(g) \otimes v + gD(v))=D_t(g)\left(\begin{matrix}f_1(v) \\ \vdots \\ f_d(v) \end{matrix}\right) +
g\left(\begin{matrix}f_1(D(v)) \\ \vdots \\ f_d(D(v)) \end{matrix}\right),
\end{gather*}
whereas
\begin{gather*}
(D\circ F)(g \otimes v)=D_{t+k}\left(g\left(\begin{matrix}f_1(v) \\ \vdots \\ f_d(v) \end{matrix}\right)\right)=
D_{t}(g)\left(\begin{matrix}f_1(v) \\ \vdots \\ f_d(v) \end{matrix}\right) +
gD_{k}\left(\begin{matrix}f_1(v) \\ \vdots \\ f_d(v) \end{matrix}\right),
\end{gather*}
so it suffices to prove that
\begin{gather*}
f_i(D(v)) = D_k(f_i(v)).
\end{gather*}
But because $f_i\in B$ this follows from the relation (\ref{13}) of Theorem~\ref{thmB}. This completes our discussion of Theorem~\ref{thmgradedZ}.
\end{proof}

\subsection{Examples}\label{SSEx}
We give here some illustrations of the previous results.

\begin{ex}
 The Virasoro element of $V[\ ]_2$ is a canonical element $\tilde{\omega} \in V_{[2]}$. According to Lemma \ref{lem4} we have $F(\tilde{\omega}) \in M_2(\rho, L)$. We will show that
\[
F(\tilde{\omega}) = D_0F_V.
\]
Indeed, $o(\tilde{\omega})= L(0)-\tfrac{c}{24}\mathbf{1}$ and we compute
\[
f_i(\tilde{\omega}) = \Tr_{M_i}\big(L(0)-\tfrac{c}{24}\Id\big)q^{L(0)-\frac{c}{24}} = D_0(f_i).
\]
The result follows.
\end{ex}

\begin{ex} Suppose that $V$ is an extremal VOA with just two simple modules and with $\ell=0$, as discussed in Section~\ref{s:rank2} below. We already observed that the minimal weight is at least zero, and since $F_V$ is of weight zero, this means the minimal weight $k_1$ of $M(\rho, L)$ is $0$ and that $F_V$ spans $M_0(\rho, L)$. Moreover, by Theorem~\ref{t:rank2}, $M(\rho, L)$ is generated by $F_V$ and $D_0F_V$ as $M$-module. By Theorem~\ref{thmgradedZ} we can deduce in this case that $F$ is a \emph{surjection}
\[
F\colon \ M \otimes V \twoheadrightarrow M(\rho, L)
\]
and especially, it follows that every holomorphic weight $k$ vector-valued modular form (vvmf) in $M(\rho, L)$ arises as $F(u)$ for some $u \in (M\otimes V[\ ])_{k}$.
\end{ex}

The conclusions of the previous example certainly do not hold in general. Analysis of the sequence
\[
0\rightarrow \ker F\rightarrow M \otimes V \rightarrow \im F\hookrightarrow \oplus_{k\geq 0} M_k(\rho, L)
\]
goes to the heart of the proof of Zhu's theorem.

Recall that $V$ is $C_2$-cofinite, meaning that the subspace $C_2(V) \subseteq V$ spanned by all states $u(-2)v$ for $u, v\in V$ has \emph{finite codimension}. We shall prove
\begin{thm}
 \label{rkthm}
 As an $M$-module, $\im F$ is generated by no more than $\codim C_2(V)$ elements.
\end{thm}
\begin{proof}
 Because $V$ and $V[\ ]$ are isomorphic VOAs, we may, and shall, work in the latter VOA and its subspace $C_2[V]$ spanned by all $u[-2]v$ for $u$, $v\in V$. First note that all states in $C_2[V]$ have square bracket weight at least $1$. Therefore, if we decompose $V$ into square bracket graded subspaces
 \begin{gather} \label{sqdecomp}
V=W \oplus C_2[V]
\end{gather}
then we necessarily have $\mathbf{1}\in W$. We will prove that
\[
M\otimes V = M \otimes W +\ker F.
\]
This immediately implies the statement of the theorem.

We proceed by establishing, by induction on $k$, that each $v \in V_{[k]}$ belongs to $M \otimes W + \ker F$. Because $\mathbf{1}\in W$ the result is obvious for $k = 0$ and this begins the induction. Thanks to equation~\eqref{sqdecomp} we can write
\begin{gather} \label{vform}
v=w+\sum_j a_j[-2]b_j,\qquad w\in W,\quad a_j, b_j \in V,
\end{gather}
with $\wt[a_j]+\wt[b_j]+1=k$ for each $j$.

The gist of the argument is the following containment, cf.~\cite{DLMModular} and Theorem~\ref{thmB}:
\begin{gather} \label{u2vform}
u[-2]v+\sum_{\ell\geq2} (2\ell-1)G_{2\ell}\otimes u[2\ell-2]v\in \ker F, \qquad u, v\in V.
\end{gather}
We apply this formula to $a_j[-2]b_j$, where each term $a_j[2\ell-2]b_j$ (for $\ell \geq 2)$ has square bracket weight $k-2\ell$. By induction each of these terms lies in $M \otimes W +\ker F$, therefore by equation~\eqref{u2vform}, so does $a_j[-2]b_j$ for each $j$. A final application of equation~\eqref{vform} now delivers the desired containment $v\in M\otimes W+\ker F$ and the theorem is proved.
\end{proof}

\begin{rmk}The reader can compare this result with \cite[Lemma~4.4.1]{Zhu}, though note that Zhu's result depends on some technical hypotheses on~$V$, some of which were eliminated in~\cite{DLMTwisted}.
\end{rmk}

\begin{rmk}We may supplement Theorem~\ref{rkthm} with the observation that all odd homogeneous subspaces $M \otimes V_{[2k+1]}$ are contained in $\ker F$. Indeed, this follows because~$F$ is a graded map into the $2\ZZ$-graded space~$M(\rho, L)$.
\end{rmk}

\subsection{Functorial properties}
It is awkward to make Zhu theory fully functorial. The reason is that many of the most important results depend on special properties of the representation $\rho$. Dependence on this particular matrix representation can ruin functoriality. In this subsection we salvage a few crumbs of functoriality.

Let $\bV$ be the category of strongly regular VOAs. A morphism $U\stackrel{f}{\longrightarrow}V$ in this category preserves vacuum vectors and conformal vectors, and satisfies $f\circ Y(u, z) = Y(f(u), z)\circ f$, that is, $f(u(n)v) = f(u)(n)f(v)$. On account of the fact that strongly regular VOAs are simple it follows that all morphisms in $\bV$ are \emph{injective}.

\begin{lem} \label{lemBres}
 The assignment $V\longrightarrow B_V$ defines a contravariant functor
\[
B\colon \ \bV\longrightarrow \Gamma-\Mod.
\]
\end{lem}
\begin{proof}
 Let $U\!\stackrel{f}{\longrightarrow}\!V$ be a morphism of strongly regular VOAs. Let $b\in B_V$. Then \mbox{$U \cong f(U) \subseteq V$} and we assert that the \emph{restriction} of $b$ to $f(U)$ belongs to $B_U$. Indeed, $b\df \sum_j c_j f_j$ is a~linear combination of the trace functions $f_j$ furnished by the simple $V$-modules $M_j$. Since $f(U)$ is strongly regular, each $M_j$ splits into a direct sum of simple $f(U)$-modules, and therefore $b$ is itself a linear combination of trace functions for the simple $U$-modules. In this way restriction to $f(U)$ defines a map
\begin{gather*}
 B(f) \df f^*\colon \ B_V\longrightarrow B_U, \\
 f^*(b) \df \Res^V_{f(U)}(b).
\end{gather*}

Now, we have to show that $f^*$ is a morphism of $\Gamma$-modules. That is, we need the following diagram to commute for $\gamma \in \Gamma$:
\[
\xymatrix{
&B_U\ar[d]_{\gamma} && B_V\ar[ll]_{f^*}\ar[d]^{\gamma}\\
&B_U && B_V\ar[ll]^{f^*}.
}
\]

Let $u\in U_{[k]}$. Then $f(u) \in V_{[k]}$. Let $b \in B_V$. We have
\begin{align*}
f^*(b \mid \gamma)(f(u), \tau)&= \Res^V_{f(U)}( b \mid \gamma )(f(u), \tau)=j(\gamma, \tau)^{-k} b(f(u), \gamma\tau)\\
&= \big(\big(\Res^V_{f(U)}b\big) \mid \gamma\big) (f(u), \tau)= (f^*(b) \mid \gamma ) (f(u), \tau),
\end{align*}
and this is the desired commutativity.
\end{proof}

\section{Symmetry and unitarity}\label{SSandU}
The matrix representation $\rho$ of $\Gamma$ with respect to the basis $(f_1, \dots, f_d)$ of the conformal block $B_V$ has a number of remarkable properties, including unitarity and symmetry of $\rho(S)$. The purpose of this section is to shed light on this situation through the optic of modularity. In order to carry this out we shall, in place of $\rho$, consider matrix representations $\sigma$ of $\Gamma$ on the space $\ch_V$ for a~strongly regular VOA~$V$. In particular, $\sigma$ is a congruence representation thanks to \cite{CDT}. Under certain conditions we will prove that $\sigma$ is itself unitary and that $\sigma(S)$ is symmetric. One of the main points is the manner in which these results are related to ideas from the theory of modular forms. For example we will resurrect an old idea of Hecke, namely his so-called operator $K$ (this is \emph{not} the $K$ in our list of notations!) and show that it is closely related to unitarity and symmetry of $S$-matrices. These results and more are explained in the next few subsections.

\subsection[Hecke's operator K and the matrix J]{Hecke's operator $\boldsymbol{K}$ and the matrix $\boldsymbol{J}$}
Let us first recall the original operator $K$ as introduced by Hecke \cite{Hecke}; see also \cite[Section~8.6]{Rankin}. It operates on meromorphic functions $f$ as follows:
 \[
f|K (\tau) \df \overline{ f(- \overline{\tau})}.
\]
 In particular, if $f$ is a modular form, then so is $f|K$ and the $q$-expansions are related as follows:
\begin{gather} \label{Kaction}
f(\tau) = \sum a_nq^{n/N}, \qquad f|K(\tau) = \sum \overline{a_n} q^{n/N}.
\end{gather}
We may, and shall, extend $K$ to vector-valued modular forms in a completely parallel manner. Then equation \eqref{Kaction} holds for every component of the vector-valued modular form.

We also make use of the integer matrix $J$ of determinant $-1$ as in the notation. Conjugation by $J$ determines an outer involutive automorphism of $\Gamma$, so we may ``twist'' any representation~$\rho$ of $\Gamma$ by $J$ to obtain a representation~$\rho^J$ defined by
\begin{gather*}
\rho^J(\gamma)\df \rho\big(J\gamma J^{-1}\big).
\end{gather*}
Note that
\begin{gather*}
J \twomat abcd J^{-1}= \twomat a{-b} {-c} d.
\end{gather*}
The complex conjugate representation $\bar{\rho}$ is defined in the obvious way, namely $\bar{\rho}(\gamma)\df\overline{\rho(\gamma)}$. The combination of these twists appears in the next Lemma, which is quite general.

\begin{lem} \label{conjliniso} Let $\rho$ be any representation of $\Gamma$. Then $K$ induces a conjugate-linear isomorphism
\begin{gather*}
 M_k^!(\rho) \stackrel{\cong}{\rightarrow} M_k^!\big(\overline{\rho}^J\big), \qquad F \mapsto F|K.
\end{gather*}
\end{lem}
\begin{proof}It will be convenient to use standard cocycle notation $j(\gamma, \tau)=c\tau+d$. Suppose that $F \in M_k^!(\rho)$. Then for $\gamma \in \Gamma$ we have
 \begin{align*}
\overline{\rho}^J(\gamma) (F|K)(\tau)& = \overline{\rho^J(\gamma) F (- \overline{\tau})} = \overline{\rho(J \gamma J^{-1}) F (- \overline{\tau})}
= \overline{ F |_k J \gamma J^{-1}(- \overline{\tau})} \\
&=\overline{ j(J \gamma J^{-1}, - \overline{\tau})^{-k} F(J \gamma J^{-1}(-\overline{\tau})) }
= j(\gamma, \tau)^{- k} \overline{F(- \overline{\gamma \tau}) } \\
&= j(\gamma, \tau)^{- k} F|K(\gamma \tau)
= (F|K)|_k \gamma (\tau).
\end{align*}
This shows that $K$ maps $M_k^!( \rho)$ into $M_k^!\big( \overline{\rho}^J\big)$. Since $K$ is an involution, it is bijective. Finally, it is obvious that $K$ is biadditive and conjugate linear in the sense that $(aF)|K = \overline{a}(F|K)$ for $a \in \CC$. This completes the proof of the lemma.
\end{proof}

To apply this result, choose any maximal set of linearly independent vectors, call it $g_1, \dots , g_m,$ from among the irreducible characters $f_1, \dots, f_d$. We write matrices and vectors with respect to this particular basis of $\ch_V$ and in particular let $\sigma$ be the matrix representation of $\Gamma$ that it furnishes.

Introduce the (column) vector-valued modular form $E\df (g_1, \dots, g_m)^{\rm T}$. It belongs to $M_0^!(\sigma)$. We now have the following quite general results:

\begin{lem}\label{lemS2id}$\sigma(S)^2=\id$.
\end{lem}
\begin{proof} We have $\sigma(S)E(\tau)= E(S\tau)$: replacing $\tau$ by $S\tau$ in this equality yields $\sigma(S)^2E=E$, hence $\sigma(S)^2=\id$ follows because the components of $E$ are linearly independent.
\end{proof}

\begin{lem}\label{lemsymmid} We have
 $\sigma=\bar{\sigma}^J$, and in particular
 \begin{gather} \label{Scond}
\sigma(S)=\overline{\sigma(S)}.
\end{gather}
Especially, $\sigma$ is a congruence representation and it is unitary representation if, and only if, $\sigma(S)$~is \emph{symmetric}.
\end{lem}
\begin{proof}Consider once again the (column) vector-valued modular form $E(\tau)\in M_0^!(\sigma)$. It has components with integral Fourier coefficients just because the components are the graded characters of $V$. As a result we have from equation~\eqref{Kaction} that $E \mid K = E$. Now by Lemma~\ref{conjliniso} we deduce that for all $\gamma\in\Gamma$ we have
 \begin{gather*}
\bar{\sigma}^J(\gamma)E= E|_0 \gamma = \sigma(\gamma)E.
\end{gather*}
Because the components of $E$ are linearly independent it follows that $\sigma= \bar{\sigma}^J$ as asserted. Now $JSJ^{-1}=S^{-1}$, and therefore $\sigma(S)=\overline{\sigma}^J(S) = \overline{\sigma(S)}^{-1}$, and then equation \eqref{Scond} follows from Lemma~\ref{lemS2id}.

Because $\sigma(T)$ is a (diagonal) unitary matrix and $\Gamma=\langle S, T\rangle$, it follows that $\sigma$ is unitary if, and only if, $\sigma(S)$ is unitary, and this is equivalent to the symmetry of $\sigma(S)$ by Lemma~\ref{lemS2id} and~\eqref{Scond}.
 \end{proof}

\section{Conformal modular forms}\label{s:conformal}
We have seen that the character vector $F_V$ of a strongly regular VOA defines a vector-valued modular form satisfying many nice properties. The notion of \emph{conformal modular form} in Definition \ref{d:conformal} codifies these desirable properties, at least in the setting where $d = \dim \ch_V$. The main ideas captured in Definition \ref{d:conformal} are that $F_V$ should have coefficients that count the dimensions of a graded vector space, and that the representation underlying $F_V$ should be equal to the representation of some modular tensor category \cite{Huang}. However, this definition should be seen more as an approximation or heuristic. It will surely require adjustment, especially as modularity results are expanded more deeply into the realm of logarithmic conformal field theory. For more on this exciting subject, which we otherwise ignore here, the reader can consult \cite{CreutzigGannonLCFT,CreutzigRidoutLCFT2013,FLohrLCFT2002,FlohrLCFT2003,FuchsLCFT2019,GaberdielLCFT,GainutdinovLCFT2013,GainutdinovLCFT2009, KawaiLCFT,Nagi2005} and the extensive lists of references found in these sources.

Before describing precisely what properties this approximation of the notion of a conformal modular form should have, we first make the technical observation that it is unambiguous to say that a weakly holomorphic modular form has nonnegative integral Fourier coefficients. Indeed, given any choice of exponents $L$ for a representation $\rho$, a weakly-holomorphic form $F \in M^!(\rho)$ can be expanded as $F(q) = q^Lf(q)$ for $f \in \lseries{\CC^d}{q}$. Different choices of $L$ correspond to adding integers to eigenvalues of $L$, and these adjust the $q$-expansion only by rescaling coordinates by integer powers of $q$. Therefore whether $f(q)$ has integer coefficients, or nonnegative coefficients, is independent of the choice of $L$.
\begin{dfn}
 \label{d:quasiconformal}
 A weakly holomorphic vector-valued modular form $F = (f_j)_{j=1}^d$ of weight zero for a representation $\rho \colon \Gamma \to \GL_d(\CC)$ is said to be a \emph{quasi-conformal modular form} provided that the following conditions are satisfied:
\begin{enumerate}\itemsep=0pt
\item[1)] the Fourier coefficients of $F$ are nonnegative integers,
\item[2)] the first nonzero Fourier coefficient of $f_1$ is $1$,
\item[3)] $\rho(S)$ is real symmetric with entries in an abelian extension of $\QQ$,
\item[4)] all entries in the first row of $\rho(S)$ are nonzero,
\item[5)] $\rho(T)$ is of finite order.
\end{enumerate}
\end{dfn}

\begin{rmk} \label{r:symmetry} Hypothesis (3) on the symmetry of $\rho(S)$ is a reasonably strong hypothesis and it could be dropped from the list of axioms. For example, it implies that every module of a~hypothetical VOA whose character vector is $F$ is self-dual. This situation is rare but does occur infinitely often. Nevertheless we (and others) have found it to be a useful hypothesis in classification problems. Note also that Ng--Wang--Wilson have shown in~\cite{NgWangWilson} that every congruence representation $\rho$ has a basis such that $\rho(T)$ is diagonal and $\rho(S)$ is symmetric, though this basis may in practice not agree with the basis of characters. This transpires for example with $A_{2,1}$, which has three simple modules but a two-dimensional space of characters. In this case the $S$-matrix acting on the natural basis of characters is
 \[
\frac{1}{\sqrt{3}}\twomat 121{-1},
 \]
 which is not symmetric. Notice though that, consistent with~\cite{NgWangWilson}, this matrix can be symmetrized while keeping the corresponding $T$-matrix diagonal. For these reasons, and with this warning, we will retain the symmetry condition above and in our discussion below.
\end{rmk}

Thanks to the recent breakthrough proof of the unbounded denominator conjecture for vector-valued modular forms, due to Calegari--Dimitrov--Tang \cite{CDT}, it follows from the axioms~(1) and~(5) above that $\ker \rho$ is a \emph{congruence subgroup}. In principle, this congruence property could have been deduced from the modular tensor category approach of \cite{DongLinNg} using Huang's results~\cite{Huang} on the representation arising from the associated conformal block, though, as noted in Remark~\ref{r:symmetry}, in general this representation can differ slightly from the representation $\rho$ defined by the characters (for example, if~$V$ has some modules that are not self-dual). The approach to the congruence nature of $\rho$ via~\cite{CDT} appears to be more direct and more general than the approach via modular tensor categories using \cite{DongLinNg, Huang}.

Suppose that $F$ is a quasi-conformal modular form for a representation $\rho$ of rank $d$, and write~$S_{ij}$ for the $(i,j)$ entry of $\rho(S)$. The \emph{fusion rules} of $\rho(S)$ are defined by
\[
 N^\nu_{\lambda\mu} = \sum_{\sigma=1}^d \frac{S_{\lambda\sigma}S_{\mu\sigma}\overline{S_{\sigma\nu}}}{S_{0\sigma}}.
\]
Since we have assumed that $S$ is real symmetric, this expression simplifies to
\[
 N^\nu_{\lambda\mu} = \sum_{\sigma=1}^d \frac{S_{\lambda\sigma}S_{\mu\sigma}S_{\nu\sigma}}{S_{0\sigma}}.
\]
For a strongly regular VOA $V$ with irreducible modules $M_1,\dots, M_d$, the fusion rule $N_{ij}^k$ calculates the multiplicity of $M_k$ in the fusion product $M_i\boxtimes M_j$ in the modular tensor category of representations of $V$. See \cite{Huang} for more details.

\begin{dfn}
 \label{d:conformal}
 A quasi-conformal modular form is said to be \emph{conformal} if the fusion rules of the underlying representation are all nonnegative integers.
\end{dfn}

To date, most attention has focused on forms of rank $\leq 3$. Here are two interesting examples in rank $4$.
\begin{ex}
 \label{ex:hardhexagon}
 The following matrices define an irreducible representation of $\Gamma$ of rank $4$:
 \begin{gather*}
 \rho(T) =\diagg\big({\rm e}^{2\pi{\rm i}/40},{\rm e}^{2\pi{\rm i} 31/40},{\rm e}^{-2\pi{\rm i}/40},{\rm e}^{2\pi{\rm i} 9/40}\big),\\
 \rho(S) =\frac{1}{4}\sqrt{1+\frac{1}{\sqrt{5}}}\left (\begin{matrix}
 2&-2&\sqrt{5}-1&1-\sqrt{5}\\
 -2&-2&\sqrt{5}-1&\sqrt{5}-1\\
 \sqrt{5}-1&\sqrt{5}-1&2&2\\
 1-\sqrt{5}&\sqrt{5}-1&2&-2
 \end{matrix}\right ).
 \end{gather*}
This is realized as the monodromy representation of the modular linear differential equation:
 \[
 D^4F -\frac{949}{7200}E_4D^2F + \frac{139}{21600}E_6DF -\frac{279}{2560000}E_4^2F = 0.
\]
A basis of solutions of this differential equation defines the coordinates of a modular form $F$ with monodromy $\rho$ and $q$-expansion of the form
\[
 F(q) = q^L\left(\begin{matrix}
 1 + q^{2} + q^{3} + 2q^{4} + 2q^{5} + 4q^{6} + 4q^{7} + 6q^{8} + 7q^{9} + \cdots \\
 1 + q + q^{2} + 2q^{3} + 2q^{4} + 3q^{5} + 4q^{6} + 5q^{7} + 7q^{8} + 9q^{9} + \cdots \\ 1 + q + q^{2} + 2q^{3} + 3q^{4} + 4q^{5} + 5q^{6} + 7q^{7} + 9q^{8} + 12q^{9} + \cdots\\
 1 + q + 2q^{2} + 2q^{3} + 3q^{4} + 4q^{5} + 6q^{6} + 7q^{7} + 10q^{8} + 12q^{9} + \cdots
 \end{matrix}\right),
\]
where $L = \diagg(1/40,31/40,-1/40,9/40)$. The sequences of Fourier coefficients of the coordinates correspond, in order, to the sequences A122134, A122130, A122129 and A122135 in the OEIS \cite{OEIS}. They are in fact partition functions that arise in the hard hexagon model of statistical mechanics, solved by Baxter in~\cite{Baxter}. In this case we have $\tilde{c}=\tfrac 35$, and by \cite[Theorem~8]{MasonLattice}, it follows that $F(q)$ is the character vector of the discrete series Virasoro VOA $\Vir(c_{3,5})$, which is discussed in~\cite{ArikeNagatomoSakai}, where they even write down the MLDE above. We will discuss some other examples of rank $4$ conformal modular forms in Section~\ref{s:rank4}. See also~\cite{GHM, HM,kaidi2021, MMS2}.
\end{ex}

\begin{rmk}
If the $S$-matrix of a modular tensor category has entries $S_{ij}$, then the \emph{pivotal dimension} of the~$j$th indexed object of the category is the ratio $\tfrac{S_{1j}}{S_{11}}$. Pivotal dimensions are known to be positive real algebraic numbers. On the other hand, in the previous example, one has $\tfrac{\rho(S)_{12}}{\rho(S)_{11}} = -1$. This shows that the hypothesis $\lambda_i > 0$ in \cite[Proposition~3.11]{DongLinNg} is necessary for identifying pivotal and quantum dimensions, even for VOAs as down-to-earth as discrete series Virasoro algebras.
\end{rmk}

Part (2) of Definition~\ref{d:conformal} insists that the first coordinate of a conformal modular form is normalized to have first nonzero Fourier coefficient equal to~$1$. On the VOA side of things, this means that $\dim V_0 = 1$. Such a restrictive condition, as far as the arithmetic of~$F_V$ goes, leads to many examples of quasi-conformal modular forms that are not conformal, as in the following example.
\begin{ex}\label{ex:quasiconformal} The following matrices define an irreducible representation of $\Gamma$ of rank~$4$:
 \begin{gather*}
 \rho(T) =\diagg\big({\rm e}^{-2\pi{\rm i} 41/40},{\rm e}^{2\pi{\rm i} 9/40},{\rm e}^{2\pi{\rm i} 31/40},{\rm e}^{2\pi{\rm i} 41/40}\big), \\
 \rho(S) =\frac{1}{4}\sqrt{1+\frac{1}{\sqrt{5}}}\left (\begin{matrix}
 2&2&\sqrt{5}-1&\sqrt{5}-1\\
 2&-2&\sqrt{5}-1&1-\sqrt{5}\\
 \sqrt{5}-1&\sqrt{5}-1&-2&-2\\
 \sqrt{5}-1&1-\sqrt{5}&-2&2
 \end{matrix}\right ).
 \end{gather*}
 This is realized as the monodromy representation of the modular linear differential equation:
 \[
 D^4F -\frac{8509}{7200}E_4D^2F + \frac{19039}{21600}E_6DF -\frac{468999}{2560000}E_4^2F = 0.
\]
A basis of solutions of this differential equation defines the coordinates of a modular form $F$ with monodromy $\rho$ and $q$-expansion of the form
\begin{gather*}
 F(q) = q^L\left(\begin{matrix}
 1 + 120786q^{2} \!+\! 14632531q^{3} \!+\! 629268246q^{4} \!+\! 15536981160q^{5} \!+\! \cdots\\
 492 \!+\! 466580q \!+\! 40164912q^{2} \!+\! 1462898532q^{3} \!+\! 32571172112q^{4} \!+\! \cdots\\
 22591 \!+\! 3863061q \!+\! 193342101q^{2} \!+\! 5227692946q^{3} \!+\! 95716064232q^{4} \!+\! \cdots \\
 99180 \!+\! 11114772q \!+\! 461579312q^{2} \!+\! 11153566692q^{3} \!+\!
189039000612q^{4} \!+\! \cdots
 \end{matrix}\right)\!,
\end{gather*}
where $L = \diagg(-41/40,9/40,31/40,41/40)$. In this case the fusion rules are all $\pm 1$ or $0$, and hence this example is quasi-conformal but not conformal. It is possible to permute coordinates to fix the sign problem with the fusion rules, but then condition~(2) of Definition~\ref{d:quasiconformal} fails. This failure can't be corrected by rescaling coordinates without introducing nontrivial denominators. Thus, this provides an example of a quasi-conformal modular form that has nothing at all to do with strongly regular VOAs, due to part (2) of Definition \ref{d:quasiconformal}.
\end{ex}

A basic problem that has received recent attention is the following.
\begin{prob}
Determine which conformal modular forms are realized as the characters of strongly regular VOAs.
\end{prob}
This problem is difficult even in rank one, where the answer is known only for small values of the central charge $c$. We refer the reader to Section~\ref{SSholVOAs} for a more detailed discussion, but every modular form $j+m$ for~$m$ a nonnegative integer is a conformal modular form according to Definition~\ref{d:conformal}, whereas there are fewer than $71$ values of $m$ that occur as the character of a~rank one (holomorphic) VOA with $c = 24$. There is no reason to expect that the situation will change in higher ranks. Nevertheless we do not currently have a single example of an irreducible representation $\rho$ with $\dim \rho \geq 2$ and a conformal modular form for $\rho$ that provably does \emph{not} correspond to a strongly regular VOA! It would be interesting to produce such an example, and we expect that one can be found in the case where $\dim \rho = 2$ and the exponents are chosen so that the minimal weight is $-6$, which forces the corresponding space of forms of weight zero to be two-dimensional. The case when $\dim \rho =2$ is discussed in more detail in Section~\ref{s:rank2} below.

In the remainder of this paper we discuss how to compute examples of conformal modular forms in arbitrary rank, and then we provide an overview of known examples in low dimensions. One open question following the breakthough of \cite{CDT} is whether their diophantine techniques can be utilized in the study of conformal specializations of families of modular forms, providing a more general suite of tools than the hypergeometric approach of papers such as \cite{FM6} and \cite{MNS}. We discuss the families of modular forms of interest next.

\section{Frobenius families}\label{s:frobenius}
A \emph{Frobenius family of modular forms}, as in Definition \ref{d:frobfam} below, is a family of modular forms that is essentially equivalent to a family of ordinary differential equations. Most VOA classification results to date have focused on classifying character vectors arising in specific Frobenius families. We will discuss that work below after briefly formalizing the notion of Frobenius family.

Since modular forms have slow growth at cusps, we focus on regular singular equations. Moreover, since we are primarily interested in vector-valued modular forms for the full modular group $\Gamma$, we focus on equations whose singularities are at $0$, $1$ and $\infty$, and whose monodromy factors through~$\Gamma$. The uniformizer $K = 1728/j$ as in the notation subsection maps $\mathcal{H} \rightarrow \PP^1\setminus \{0,1,\infty\}$, where the orbit of the cusp maps to $0$, the orbit of ${\rm i}$ maps to $1$, and the orbit of the other elliptic point maps to $\infty$. We shall use $K$ as the local coordinate at $0$ and write $\theta = K{\rm d}/{\rm d}K$. The only advantage provided by $K$ over other uniformizing maps is that much of the classical literature on differential equations assumes that $0$, $1$ and $\infty$ are among the singularities of an equation with at least three singularities, so that many classical formulas are most naturally expressed in terms of~$K$.

As discussed in~\cite{FM2}, pulling back regular equations on $\PP^1\setminus \{0,1,\infty\}$ via the uniformizing map~$K$ yields modular linear differential equations on the complex upper half-plane (see also~\mbox{\cite{BantayGannon, Gannon}}). If
\begin{enumerate}\itemsep=0pt
\item[1)] the monodromy group factors through $\Gamma$, and
\item[2)] the solutions are in fact holomorphic at elliptic points,
\end{enumerate}
then the solutions are vector-valued modular forms, possibly with a pole of finite order at the cusp.

Before we discuss modular aspects of ordinary differential equations, we shall recall some classical results. During this general discussion we work with a parameter $z$ on $\PP^1$ and let $w = 1/z$ be a coordinate at infinity. Later we will set $z = K$.

A regular singular ordinary differential equation on $\PP^1\setminus \{0,1,\infty\}$ of degree $d \geq 1$ takes the form
\begin{equation}\label{eq1}
\left(\sum_{j=0}^d(1-z)^jP_{d-j}(z)\theta^j\right)f = 0,
\end{equation}
where $\theta = z{\rm d}/{\rm d}z$, $P_j(z) \in \CC[z]$ satisfies $\deg P_j \leq j$ for all~$j$, and $P_0(z) = 1$. The \emph{indicial polynomial at $z=0$} of equation~\eqref{eq1} is
\[
 R_0(X) = \sum_{j=0}^d P_{d-j}(0)X^j.
\]
Indicial polynomials $R_1(X)$ and $R_\infty(X)$ at $1$ and $\infty$, respectively, are defined similarly by expressing equation~\eqref{eq1} in terms of the parameters $z-1$ and $w = 1/z$. The roots of these indicial polynomials are called \emph{local exponents} at the corresponding points. If $e$ is one of these roots, then $\lambda = {\rm e}^{2\pi{\rm i} e}$ is an eigenvalue of the corresponding local monodromy transformation. If the exponents at a point $z = p$ do not differ by integers, then the corresponding local monodromy transformation is diagonalizable and its conjugacy class is determined uniquely by~$R_p(X)$.
\begin{ex} \label{ex:hyper} The classical Gauss hypergeometric equation is the differential equation
 \[
y'' + \frac{\gamma - (\alpha+\beta+1) z }{z(1-z)}y' + \frac{-\alpha\beta}{z(1-z)}y = 0.
\]
The indicial polynomials are
\begin{gather*}
 R_0(X) =X(X-1) +\gamma X = X(X+\gamma-1),\\
 R_1(X) =X(X-1)-(\gamma-\alpha-\beta-1)X = X(X+\alpha+\beta-\gamma),\\
 R_{\infty}(X) = X(X-1)-(-\alpha-\beta-1)X +\alpha\beta = (X-\alpha)(X-\beta).
\end{gather*}
\end{ex}

Next we recall the Frobenius method for solving an equation such as \eqref{eq1}. For simplicity in this paper we restrict to the case where the exponents $e_1,\dots, e_d$ at $z=0$ are all distinct mod~$\ZZ$. Let $e = e_j$ be any one of these exponents. The Frobenius method proceeds by searching for solutions of the form $f(z) = z^e\sum_{n\geq 0} a_nz^n$. The main theorem is that such solutions converge in a neighbourhood of $z=0$. In our setting they converge in the region $\abs{z} < 1$.

We now set $z=K$ and make the following hypotheses on equation \eqref{eq1}:
\begin{enumerate}\itemsep=0pt
\item[1)] the monodromy representation factors through $\SL_2(\ZZ)$,
\item[2)] the exponents at $K=1$ and $K= \infty$ are nonnegative,
\item[3)] the exponents at $K=0$ are rational numbers that are distinct mod $\ZZ$.
\end{enumerate}
Condition (1) ensures that a basis of solutions of \eqref{eq1} defines a vector-valued modular form for the monodromy representation, but possibly with poles at the cusp and elliptic points. Condition~(2) ensures that the solutions are in fact holomorphic functions on $\uhp$. Condition~(3) allows us to use the classical Frobenius method to compute a basis of solutions near $K=0$ without having to worry about logarithmic terms. Note too that these conditions allow for the solutions to have poles at the cusp if the exponents at the cusp are negative, and for applications to vertex operator algebras, we have already observed that negative exponents at $K=0$ are unavoidable.

\begin{rmk}The differential equations satisfied by character vectors $F_V$ for strongly regular VOAs $V$ satisfy conditions (1), (2) and (3) above. Conditions~(1) and~(2) are a consequence of Zhu's theorem. The rationality of the exponents at $K=0$ is a consequence of a result of Anderson--Moore~\cite{AndersonMoore} or~\cite {DLMModular} which imply that $\rho(T)$ is of finite order for strongly regular $V$. Finally, Tuba--Wenzl~\cite{TubaWenzl} classified the irreducible representations of $\Gamma$ up to rank~$5$, and all such $\rho$ with $\rho(T)$ of finite order have distinct eigenvalues, so that condition~(3) is satisfied in these cases. In general one can and probably should discard condition~(3), but we will keep it in force in this paper for simplicity, since it holds in the cases that we consider below.
\end{rmk}

Recall that when the exponents at $K=0$ are distinct mod $\ZZ$, for each exponent $e_j$ there is a recursively computable solution of the form
\begin{gather*}
 f_j(K) = K^{e_j}\sum_{n \geq 0}a_nK^n,
\end{gather*}
where we are free to take $a_0 = 1$. Write $(1-K)^nP_n(K) = \sum_{k=0}^d p_{nk}X^k$ for scalars $p_{nk}$, and set
\[
Q_k(X) = \sum_{n=0}^d p_{nk}X^n.
\]
Observe that $Q_0(X) = R_0(X)$. Starting from the identity
\[
\theta^m(f_j(K)) = K^{e_j}\sum_{n\geq 0}(e_j+n)^ma_nK^n,
\]
a straightforward computation shows that $f_j(K)$ is a solution to equation \eqref{eq1} if and only if the coefficients $a_n$ satisfy the following recursive formula for $n \geq 1$:
\begin{equation} \label{eq:recurrence}
 a_n = -\sum_{k=1}^{\min(d,n)}\frac{Q_k(e_j+n-k)}{Q_0(e_j+n)}a_{n-k}.
\end{equation}
Observe that our hypotheses on the exponents ensures that $Q_0(e_j+n)$ is never zero for $n \geq 1$.

All of this material is rather classical. We now examine some features regarding how the Frobenius solutions behave in families. That is, we now suppose that the coefficients of equation~\eqref{eq1} satisfy $P_n(K)\in \QQ[e_1,\dots, e_d][K]$. In fact, after rescaling, we may assume that $P_n(K) \in \ZZ[ e_1,\dots, e_d][K]$. Then we also have $Q_k(X) \in \ZZ[e_1,\dots, e_d][X]$ for all $k$. Therefore, if we normalize the solutions so that $a_0 = 1$, the solutions to the recurrence relation~\eqref{eq:recurrence} are rational functions in $\QQ(e_1,\dots, e_d)$.
\begin{ex}
From \cite[Section~4]{FM2}, the monic modular linear differential equation of degree~$2$ and weight~$0$ corresponds to the ODE
 \[
\big((6-6K)\theta^2 - (2K + 1)\theta + 6\alpha\big)f =0,
\]
where $\alpha = e_1e_2 \in \QQ[e_1,e_2]$ and $e_1+e_2 = \tfrac{1}{6}$. We have
\begin{gather*}
 Q_0(X) = 6\big(X^2-\tfrac 16 X + \alpha\big),\qquad
 Q_1(X) = -6X\big(X+\tfrac{1}{3}\big),\qquad
 Q_2(X) = 0.
\end{gather*}
Since $Q_2(X) = 0$, the recurrence relation \eqref{eq:recurrence} relative to the exponent $e_1$ (say) is the hypergeometric relation
\[
 a_n = -\frac{Q_1(e_1+n-1)}{Q_0(e_1+n)}a_{n-1}=\frac{(e_1+n-1)\big(e_1+n-\tfrac 23\big)}{(e_1-e_2+n)n}a_{n-1} = \frac{\big(e_1)_n(e_1+\tfrac 13\big)_n}{(e_1-e_2+1)_nn!},
\]
where $(b)_n = b(b+1)\cdots (b+n-1)$ is the rising factorial. Therefore, in this case a basis of solutions is given by the coordinates of the series
\[
 F(K) \df \twovec{f_1(K)}{f_2(K)} = \left(\begin{matrix} K^{e_1}{}_2F_1\big(e_1,e_1+\tfrac 13; e_1-e_2+1;K\big)\vspace{1mm}\\
 K^{e_2}{}_2F_1\big(e_2,e_2+\tfrac 13; e_2-e_1+1;K\big)\end{matrix}\right).
\]
This is an example of a family of modular forms, and it is the prototype for the notion of Frobenius family below.
\end{ex}

This example, and a similar computation in degree $3$, allow one to describe essentially all vector-valued modular forms in ranks $2$ and $3$. The situation changes once one turns to forms of rank $4$. In this case the recurrence relation~\eqref{eq1} does not simplify and, in particular, is not hypergeometric. For example, it is clear by group theory alone that $_4F_3$ cannot be used directly to describe vector-valued modular forms of rank~$4$ for $\SL_2(\ZZ)$: the irreducible representations of $\SL_2(\ZZ)$ of rank $4$ have local exponents equal to $0$, $0$, $\tfrac 12$, $\tfrac 12\pmod{\ZZ}$ at $K=1$, while the local exponents of $_4F_3$ are $0$, $0$, $0$, $\tfrac 12 \pmod{\ZZ}$. This incompatibility ensures that $_4F_3$ does not help directly in rank~$4$, and computations suggest that naive use of more general hypergeometric series is not immediately helpful either. Therefore it is desirable to introduce more general methods.

To describe what we mean by a Frobenius family of modular forms, let $e_1$ through $e_d$ denote coordinates on affine space, and set
\[
 X = \big\{(e_1,\dots, e_d)\in \CC^d \mid e_i - e_j \not \in \ZZ \textrm{ for all } i\neq j\big\}.
\]
This set is the parameter space for Frobenius families. We will be mostly interested in the rational points of this space.
\begin{dfn} \label{d:frobfam}
 For polynomials $P_n(K) \in \ZZ[e_1,\dots, e_d][K]$ of degree $\leq d$ in $K$ defining a regular singular ODE of the form~\eqref{eq1}, the \emph{Frobenius family} of solutions is the formal vector-valued series
 \[
 F(K) = \left(\begin{matrix}
f_1(K)\\ f_2(K)\\ \vdots\\ f_d(K)
 \end{matrix}
 \right),
\]
where each $f_j(K) \in K^{e_j}\pseries{\QQ(e_1,\dots, e_d)}{K}$ is defined by the recurrence relation~\eqref{eq:recurrence}, normalized so that $f_j(K) = K^{e_j}(1 + O(K))$.
\end{dfn}

\begin{rmk}We stress that a Frobenius family is expressed most naturally in terms of a~$K$-expansion. The $q$-expansion can be extracted from this by substituting the $q$-expansion for $K = 1728/j$ and using the binomial formula to expand the terms $K^{e_j}$. In particular, if the $K$-expansion coefficients are rational, then the $q$-expansion coefficients are contained in $\QQ(1728^{e_j})$. On the other hand, the $q$-expansion of $F(K)$ could be integral with the $K$-expansion being non-integral.
\end{rmk}

\begin{rmk}The definition of a Frobenius family above may be slightly too restrictive. For example, if one wants to consider the family $j+m$ for varying $m\in \CC$ as a Frobenius family of modular forms, it corresponds to the family of ODEs $\frac{{\rm d}f}{{\rm d}j} = \frac{1}{j+m}f$ which has a moving apparent singularity at $j = -m$. Therefore one may want to allow the differential equation~\eqref{eq1} to have a finite number of additional apparent singularities.
\end{rmk}

If $F$ is a Frobenius family, then for each exponent tuple $(e_j) \in X$, the specialization of~$F$ converges to a local solution of the corresponding ODE near $K=0$. When the (global) monodromy representation of this specialization factors through $\SL_2(\ZZ)$ for each point of some open set $U \subseteq X$, and the specializations in~$U$ are all holomorphic at $1$ and $\infty$, then we say that $F$ is a \emph{Frobenius family of modular forms}.
\begin{lem} \label{l:frobfam}
 Let $F = (f_j)$ be a Frobenius family of modular forms, where
 \[ f_j(K) = K^{e_j}\sum_{n\geq 0}a_{jn}K^n.\]
Let $N$ be the maximal degree in $\QQ[e_1,\dots, e_d]$ of all coefficients of the $P_n(K)$ of equation~\eqref{eq1}. Then the following properties hold:
 \begin{enumerate}\itemsep=0pt
 \item[$(a)$] each coefficient $a_{jn} \in \QQ(e_1,\dots, e_d)$ is a ratio of polynomials of degree $\leq (N+d)^n$,
 \item[$(b)$] for each $j$ and for each $n \geq 1$,
 \[
 \bigg(n!\prod_{i\neq j}(e_j-e_i+1)_n\bigg) a_{jn} \in \ZZ[e_1,\dots, e_d].
 \]
 \end{enumerate}
\end{lem}
\begin{proof}
 For (a), first observe that by equation \eqref{eq:recurrence} we have
 \[
 a_{j1} = -\frac{Q_1(e_j)}{Q_0(e_j+1)}.
\]
The numerator and denominator both have degree at most $N+d$, which proves part (a) in this case. The general case follows by induction. Part (b) also follows by \eqref{eq:recurrence} and induction, since
\begin{align*}
 \left(n!\prod_{i\neq j}(e_j-e_i+1)_n\right) a_{jn} &= -\left(\prod_{i=1}^d(e_j-e_i+1)_n\right)\sum_{k=1}^{\min(d,n)}\frac{Q_k(e_j+n-k)}{\prod_{i=1}^d(e_j-e_i+n)}a_{j,n-k}\\
 &=-\sum_{k=1}^{\min(d,n)}Q_k(e_j+n-k)\left(\prod_{i=1}^d(e_j-e_i+1)_{n-1}\right)a_{j,n-k}.\tag*{\qed}
 \end{align*}\renewcommand{\qed}{}
\end{proof}

\begin{rmk}Part (b) of Lemma~\ref{l:frobfam} can be used to find explicit denominators for the coefficients of a Frobenius family. In practice, though, there tends to be significant cancellation among the numerators and denominators in the recursive computation of the coefficients $a_{jn}$, and this bound on the denominators tends to be quite poor for arithmetic applications.
\end{rmk}
\begin{ex}\label{ex:rank4} In~\cite{FM2} it was shown that the general monic MLDE of weight $0$ and degree $4$ corresponds to the ODE defined by the following differential operator:
\begin{gather*}
\big(K^2-2K+1\big)\theta^{4} + \big(2K^2-K-1\big) \theta^{3} + \big(\tfrac{11}{9} K^{2} - \tfrac{9 a+7}{9}K + \tfrac{36 a + 11}{36}\big) \theta^{2} \\
 \qquad{}+ \big(\tfrac{2}{9} K^{2} - \tfrac{3 a + 9 b + 1}{9}K +\tfrac{- 6 a + 36 b - 1}{36}\big) \theta + c.
\end{gather*}
We have the relations
\begin{gather*}
 1 = e_1+e_2+e_3+e_4,\qquad
 a = \sigma_2(e_1,e_2,e_3,e_4)-\tfrac{11}{36},\\
 b = \tfrac{1}{6}a+\tfrac{1}{36}-\sigma_3(e_1,e_2,e_3,e_4),\qquad
 c = e_1e_2e_3e_4,
\end{gather*}
where the $\sigma_n(e_1,\dots, e_d)$ denote elementary symmetric polynomials of degree $n$. In this case
\begin{align*}
 Q_0(X) &= X^4-X^3+\big(a+\tfrac{11}{36}\big)X^2+\big(b-\tfrac{1}{6}a-\tfrac{1}{36}\big)X+c\\
 &= (X-e_1)(X-e_2)(X-e_3)(X-e_4),\\
 Q_1(X) &= -X\big(2X^3+X^2-\big(\tfrac 79-a\big)X+\tfrac{3a+9b+1}{9}\big),\\
 Q_2(X) &= \tfrac{1}{9}X(X+1)(3X+1)(3X+2).
\end{align*}
We obtain the following recursive identities for the $K$-series coefficients of the Frobenius family:
\begin{gather*}
 a_{j1} = -\tfrac{Q_1(e_j)}{\prod_{i\neq j}(e_j-e_i+1)},\qquad
 a_{jn} = -\tfrac{Q_1(e_j+n-1)}{Q_0(e_j+n)}a_{n-1}-\tfrac{Q_2(e_j+n-2)}{Q_0(e_j+n)}a_{n-2}.
\end{gather*}
Therefore if we define a matrix sequence
\[
 M_{jn} = \frac{1}{Q_0(e_j+n)}\twomat{-Q_1(e_j+n-1)}{-Q_2(e_j+n-2)}{Q_0(e_j+n)}{0}
\]
then we find by recursion that
\[
\twovec{a_{jn}}{a_{j,n-1}} = M_{jn}\cdots M_{j2}M_{j1}\twovec 10.
\]
Unfortunately, in general this sequence is rather difficult to analyze, and the degrees of the coefficients as rational functions in the exponents grow rather quickly.
\end{ex}

While it is rarely possible to compute exact expressions for the Fourier coefficients of a~Frobenius family of modular forms, they can be computed recursively, and sometimes knowing only a~few terms is enough to give nontrivial information about any conformal modular forms that are specializations of the Frobenius family. This is because conformal modular forms have coefficients that are positive integers, and the sign of each coefficient $a_{jn}$ is constant on the connected components of the divisor of $a_{jn}$, regarded as a rational function of the exponents $e_1,\dots, e_d$. Therefore, this limits the search for conformal specializations of a~Frobenius family to subsets of affine space defined by simple algebraic inequalities. Using only the \emph{first} coefficient of a Frobenius family of rank three allowed the authors in~\cite{FM6} to reduce the classification problem for conformal modular forms in a two-parameter family to a far more manageable computation. It is clear that this technique has a wider range of applications, and we give some new examples below in Section~\ref{s:rank4}.

\begin{rmk}We have focused on the $K$-series expansion of a Frobenius family, normalized so that the $j$th coordinate takes the form $K^{e_j}(1+O(K))$. This normalization of the Frobenius family then leads to a $q$-series of the form
 \[
 f_j(q) = (1728)^{e_j}q^{e_j} + O\big(q^{e_j+1}\big).
\]
In this basis, the monodromy representation rarely has $\rho(S)$ symmetric. As we will show in Theorem~\ref{t:rank2} in Section~\ref{s:rank2}, symmetrizing the $S$-matrix can be a challenging computation that typically requires one to introduce some analytic factors that are not algebraic.
\end{rmk}

\section{Strongly regular VOAs with two modules}\label{s:rank2}
To describe character vectors of strongly regular VOAs with two simple modules it is helpful to first recall the structure of modules of vector-valued modular forms in rank $2$. In the following theorem we add precision to some past work~\cite{MNS} by choosing a~basis so that $\rho(S)$ is a~symmetric matrix. This new feature is obtained from classical formulae for the monodromy of hypergeometric series.

\begin{thm} \label{t:rank2} Let $\rho$ be an irreducible representation of $\Gamma$ of rank $2$ with $\rho(T)$ diagonal, and let $L =\diagg(e_1,e_2)$ be exponents for $\rho(T)$. Assume that $e_1$, $e_2$ and $e_1-e_2$ are not integers. Then the minimal weight for $(\rho,L)$ is $k_1 = 6\Tr(L)-1$ and there exists a basis for $\rho$ such that
 \begin{gather*}
 \rho(T)= \diagg\big({\rm e}^{2\pi{\rm i} e_1}, {\rm e}^{2\pi{\rm i} e_2}\big),\\
 \rho(S) = \big({\rm e}^{2\pi{\rm i} (2e_1+e_2)} - {\rm e}^{2\pi{\rm i}(e_1+2e_2)}\big)^{-1}\!\twomat{1}{\sqrt{1-2\cos(2\pi (e_1-e_2))}\!}{\!\!\sqrt{1-2\cos(2\pi(e_1-e_2))}}{-1}\!.
 \end{gather*}
 In this basis, the space $M(\rho,L)$ of holomorphic modular forms for $\rho$ relative to $L$ is free of rank~$2$ over $M = \CC[E_4,E_6]$, with an explicit free-basis given by the forms:
 \begin{gather*}
 F =\eta^{2k_1}\left(\begin{matrix}{j^{-f_1} {}_2F_1\big(f_1,f_1+\tfrac 13; f_1-f_2+1;\tfrac{1728}{j}\big)}\\ {1728^{f_2-f_1}Xj^{-f_2}{}_2F_1\big(f_2,f_2+\tfrac 13; f_2-f_1+1;\tfrac{1728}{j}\big)}\end{matrix}\right),\qquad D_{k_1}F,
 \end{gather*}
 where $f_j = e_j-\tfrac{k_1}{12}$ and
 \begin{gather*}
 X = \frac{\Gamma(f_1-f_2)\Gamma(1-f_1)\Gamma\big(\tfrac 23 -f_1\big)}{\Gamma(f_2-f_1)\Gamma(1-f_2)\Gamma\big(\tfrac 23 - f_2\big)}\sqrt{-\frac{\sin(\pi f_1)\sin\big(\pi\big(f_1 + \tfrac 13\big)\big)}{\sin(\pi f_2)\sin\big(\pi\big(f_2 + \tfrac 13\big)\big)}}.
\end{gather*}
\end{thm}
\begin{rmk}
In Theorem~\ref{t:rank2} there are two choices of $S$-matrix, corresponding to two possible square roots in the expressions for~$\rho(S)$ and~$X$. The two $S$-matrices are related by conjugation with~$\stwomat {-1}0{\hphantom{-}0}1$.
\end{rmk}
\begin{proof}
Recall first that for the classical hypergeometric equation
\[
 y'' + \frac{(a+b+1)z-c}{z(z-1)}y' + \frac{ab}{z(z-1)}y = 0,
\]
in the basis of solutions $u_1 = {}_2F_1(a,b;c;z)$ and $u_2 = z^{1-c}{}_2F_1(a-c+1,b-c+1;2-c;z)$ near $z=0$ (in the notation of~\cite{IKSY} we have $u_1 = f_0(x;0)$ and $u_2 = f_0(x;1-c)$) the local monodromy matrices around $0$ and $\infty$ are
\begin{gather*}
A_0 = \twomat{1}{0}{0}{{\rm e}^{-2\pi{\rm i} c}}, \qquad A_\infty =P^{-1}\twomat{{\rm e}^{-2\pi{\rm i}a}}{0}{0}{{\rm e}^{-2\pi{\rm i}b}}P,
\end{gather*}
where
\begin{gather*}
 P = \twomat{\gamma(a,b,c)}{\gamma(b,a,c)}{\gamma(a',b',c')}{\gamma(b',a',c')},\qquad
 \gamma(x,y,z) = {\rm e}^{-\pi{\rm i} x}\frac{\Gamma(z)\Gamma(y-x)}{\Gamma(y)\Gamma(z-x)},\\
 a' = a-c+1,\qquad
 b' = b-c+1,\qquad
 c' = 2-c.
\end{gather*}
(cf.\ \cite[Theorem~4.6.2]{IKSY}). Our basis of modular forms will correspond to the basis of solutions
\begin{gather*}
 v_1 = z^{-a}{}_2F_1(a,a-c+1;a-b+1;1/z) = f_\infty(z;a),\\
 v_2 = z^{-b}{}_2F_1(b,b-c+1;b-a+1;1/z) = f_\infty(z,b)
\end{gather*}
with $z = j/1728$, and so to obtain a formula for the monodromy, we need to change basis above. But by \cite[Theorem~4.6.1]{IKSY}, the change of basis formula is
\[
 \twovec{u_1}{u_2} = P\twovec{v_1}{v_2}.
\]
Therefore, in the basis of solutions defined by $v_1$ and $v_2$, we have the monodromy matrices $B_0 = P^{-1}A_0P$ and $B_\infty = \diagg\big({\rm e}^{-2\pi{\rm i} a}, {\rm e}^{-2\pi{\rm i}b}\big)$ (notice that since we work with column vectors, our formulae differ slightly from~\cite{IKSY} by a transpose).

Now let $F \in M(\rho,L)$ be of minimal weight. It is known \cite{FM4} that the minimal weight is $k_1 = 6\Tr(L)-1$, and $G = \eta^{-2k_1}F$ has coordinates that are a basis of solutions to the differential equation
\[
(6-6K)\theta^2-(2K+1)\theta+6\alpha)f = 0.
\]
Here $K = 1728/j$, $\theta =Kd/dK$ and $\alpha = f_1f_2$ where we write $f_j = e_j - \tfrac{k_1}{12}$ (cf.\ \cite[equation~(11)]{FM2}). The Frobenius family of solutions to this family of differential equations is
\[
 G \df \left(\begin{matrix} \big(\tfrac{1728}j\big)^{f_1}{}_2F_1\big(f_1,f_1+\tfrac 13; f_1-f_2+1;\tfrac{1728}{j}\big)
 \vspace{1mm}\\ \big(\tfrac{1728}j\big)^{f_2}{}_2F_1\big(f_2,f_2+\tfrac 13; f_2-f_1+1;\tfrac{1728}{j}\big)\end{matrix}\right) = \twovec{f_\infty(j/1728;f_1)}{f_\infty(j/1728;f_2)},
\]
where we have the identities $a = f_1$, $b = f_2$ and $c = \tfrac{2}{3}$. From the expression for $G$ in terms of $f_\infty$, we find that $G$ has monodromy around $0$ and $\infty$ also given by $B_0$ and $B_\infty$. These will correspond, up to twisting by
$\psi$ as in (\ref{psidef}), to $\rho(R)^{\pm 1}$ and $\rho(T)^{\pm 1}$, where the sign in the exponent is determined by the orientations used here and in \cite{IKSY}. By comparing $q$-expansions we find that
\[
 G(\tau+1) = \twomat{{\rm e}^{2\pi{\rm i} f_1}}{0}{0}{{\rm e}^{2\pi{\rm i} f_2}}G(\tau) =B_\infty^{-1}G(\tau).
\]
That is, if we write $\rho' = \rho \otimes \psi^{-k_1}$, then $\rho'(T) = B_\infty^{-1}$. Since we have $S = \stwomat 0{-1}{1}{\hphantom{-}0}$ and $R = ST = \stwomat{0}{-1}{1}{\hphantom{-}1}$, we likewise obtain the formula
\[
 \rho(S) = (-{\rm i})^{k_1}B_0B_{\infty}^{-1} = (-{\rm i})^{k_1}P^{-1}\twomat{1}{0}{0}{{\rm e}^{2\pi{\rm i}/3}}P\twomat{{\rm e}^{2\pi{\rm i} a}}{0}{0}{{\rm e}^{2\pi{\rm i} b}}
\]
in the basis corresponding to the modular form $F = \eta^{2k_1}G$.

It turns out that $\rho(S)$ is rarely symmetric in this basis, and so our next goal is to change basis so that~$\rho(S)$ is symmetric. This amounts to rescaling the second coordinate of $F$. To determine the rescaling factor we shall simplify our expression for~$\rho(S)$. First note that, by the functional equation $\Gamma(z+1)=z\Gamma(z)$ and the reflection formula $\Gamma(1-z)\Gamma(z) = \frac{\pi}{\sin(\pi z)}$ for $z \not \in \ZZ$, one finds
\begin{gather*}
 \det P = {\rm e}^{-\pi{\rm i} (a+b-c+1)}\frac{1-c}{a-b}\frac{\sin(\pi b)\sin(\pi(c-a))-\sin(\pi a)\sin(\pi(c-b))}{\sin(\pi c)\sin(\pi(b-a))}
 = \frac{1}{3{\rm i}(a-b)},
\end{gather*}
where we have used the facts that $c = \tfrac 23$ and $a+b = f_1+f_2 = \tfrac 16$. Therefore we find that, with $\zeta = {\rm e}^{2\pi{\rm i}/3}$,
\begin{gather*}
\rho(S) = 3{\rm i}(a-b)(-{\rm i})^{k_1}\!\twomat{\gamma(b',a',c')}{\!\!\!-\gamma(b,a,c)}{-\gamma(a',b',c')}{\!\!\!\gamma(a,b,c)}\!\twomat{\gamma(a,b,c)}{\!\!\!\gamma(b,a,c)}{\zeta\gamma(a',b',c')} {\!\!\!\zeta\gamma(b',a',c')} \! \twomat{{\rm e}^{2\pi{\rm i} a}}{\!\!\!0}{0}{\!\!\!{\rm e}^{2\pi{\rm i} b}}\\
 = \twomat{\gamma(b',a',c')\gamma(a,b,c)-\zeta\gamma(b,a,c)\gamma(a',b',c')}
{\!\!\!\!(1-\zeta)\gamma(b,a,c)\gamma(b',a',c')}{(\zeta-1)\gamma(a,b,c)\gamma(a',b',c')}{\!\!\!\!\zeta\gamma(a,b,c)\gamma(b',a',c')-\gamma(a',b',c')\gamma(b,a,c)}\\
 \quad {}\times 3{\rm i}(a-b)(-i)^{k_1}\twomat{{\rm e}^{2\pi{\rm i} a}}{0}{0}{{\rm e}^{2\pi{\rm i} b}}.
\end{gather*}
Since $F$ transforms as $F(S\tau) = (-1/\tau)^{k_1}\rho(S)F(\tau)$, if we write $F' = (F_1,XF_2)^{\rm T}$, where $F = (F_1,F_2)^{\rm T}$ and $X \in \CC$, then
\[
 F'(S\tau) =\twomat 100X F(S\tau) = (-1/\tau)^{k_1}\twomat 100X\rho(S) \twomat 100{X^{-1}}F'(\tau).
\]
Since for any scalars $u_j$ we have
\[
\twomat 100X\twomat{u_1}{u_2}{u_3}{u_4} \twomat 100{X^{-1}} = \twomat{u_1}{u_2X^{-1}}{Xu_3}{u_4}
\]
the two values of $X$ making $F'$ transform under a symmetric $S$-matrix are determined by the equation
\[
 X^2 = -{\rm e}^{2\pi{\rm i}(b- a)}\frac{\gamma(b,a,c)\gamma(b',a',c')}{\gamma(a,b,c)\gamma(a',b',c')}.
\]
Once the symmetrization is performed, it is then a straightforward but tedious computation to check that the $S$-matrix and the $X$ term simplify as described in the theorem. Finally, note that $F'$ has a $q$-expansion of the form
\[
 F'(q) = \twovec{(1728)^{e_1}q^{e_1}+O\big(q^{e_1+1}\big)}{(1728)^{e_2}Xq^{e_2}+O\big(q^{e_2+1}\big)}.
\]
Thus, rescaling both coordinates by $(1728)^{-e_1}$ yields the result.
\end{proof}

We shall now explain how to apply Theorem~\ref{t:rank2} to the study of strongly regular VOAs~$V$ with exactly two nonisomorphic simple modules, such that $\Gamma$ acts irreducibly on~$\ch_V$. Let representatives for the modules be~$V$ itself and~$M$, and let~$M_0$ be the smallest nonzero graded piece of~$M$. Then the character vector $F_V$ of $V$ has a $q$-expansion of the form
\[
F_V(q) = q^{-\tfrac{c}{24}}\twovec{1+O(q)}{(\dim M_0)q^h + O\big(q^{h+1}\big)},
\]
where $c$ and $h$ are the central charge and conformal weight of $V$ and $M$, respectively. It follows that the underlying monodromy representation $\rho$ has exponents $e_1 = -\tfrac{c}{24}$ and $e_2 = h-\tfrac{c}{24}$. The minimal weight for this representation and choice of exponents is thus:
\[
 k_1 = 6(e_1+e_2)-1 = 6h-\tfrac{c}{2}-1.
\]
Note also that in the notation of Theorem~\ref{t:rank2},
\begin{gather*}
 f_1 = \tfrac{1}{12}-\tfrac{h}{2},\qquad
 f_2 = \tfrac{1}{12}+\tfrac{h}{2}.
\end{gather*}

The formula above demonstrates that the minimal weight $k_1$ is exactly the quantity $-\ell$ studied in \cite{GHM, TG, HM, kaidi2021, MMS2} and elsewhere. In this light, the extremality condition of \cite{TG} says that $-4 \leq k_1 \leq 0$. From the perspective of vector-valued modular forms, extremality in this sense corresponds precisely to the cases where $\dim M_0(\rho,L) = 1$, by Theorem~\ref{t:rank2}. In this way, Theorem~\ref{t:rank2} gives exact formulas for the character vectors $F_V$ when $V$ is extremal with two simple modules and an irreducible action of $\Gamma$ on $\ch_V$. Table~\ref{tablerank2} on page~\pageref{tablerank2} contains the result.

\begin{table}\centering
 \begin{tabular}{|c|c|}
 \hline
 $k_1$ & $F_V$ \\ \hline
 &\\
 $0$ & $\left(\begin{matrix}{j^{\tfrac{h}{2}-\tfrac{1}{12}} {}_2F_1\big(\tfrac{1}{12}-\tfrac{h}{2},\tfrac{5}{12}-\tfrac{h}{2}; 1-h;\tfrac{1728}{j}\big)}\\{1728^{h}Xj^{-\tfrac{h}{2}-\tfrac{1}{12}}{}_2F_1\big(\tfrac{1}{12}+\tfrac{h}{2},\tfrac{5}{12}+\tfrac{h}{2}; 1+h;\tfrac{1728}{j}\big)}\end{matrix}\right)$\\
 &\\\hline
 &\\
 $-2$ & $\frac{12}{1-6h}\eta^{-4}D_0\left(\begin{matrix}{j^{\tfrac{h}{2}-\tfrac{1}{12}} {}_2F_1\big(\tfrac{1}{12}-\tfrac{h}{2},\tfrac{5}{12}-\tfrac{h}{2}; 1-h;\tfrac{1728}{j}\big)}\\{1728^{h}Xj^{-\tfrac{h}{2}-\tfrac{1}{12}}{}_2F_1\big(\tfrac{1}{12}+\tfrac{h}{2},\tfrac{5}{12}+\tfrac{h}{2}; 1+h;\tfrac{1728}{j}\big)}\end{matrix}\right)$\\
 &\\\hline
 &\\
 $-4$ &$E_4\eta^{-8}\left(\begin{matrix}{j^{\tfrac{h}{2}-\tfrac{1}{12}} {}_2F_1\big(\tfrac{1}{12}-\tfrac{h}{2},\tfrac{5}{12}-\tfrac{h}{2}; 1-h;\tfrac{1728}{j}\big)}\\{1728^{h}Xj^{-\tfrac{h}{2}-\tfrac{1}{12}}{}_2F_1\big(\tfrac{1}{12}+\tfrac{h}{2},\tfrac{5}{12}+\tfrac{h}{2}; 1+h;\tfrac{1728}{j}\big)}\end{matrix}\right)$\\
&\\\hline
\end{tabular}
\caption{Extremal character vectors in rank $2$. The sign of $X$ from Theorem~\ref{t:rank2} is chosen to ensure that the Fourier coefficients are positive.}\label{tablerank2}
\end{table}

Notice that in the extremal cases we can deduce an analytic formula for $\dim_{\CC} M_0$, where recall that here $M_0$ denotes the smallest graded piece of the nonadjoint module of~$V$.\footnote{In particular, here $\dim_\CC M_0$ should not be confused with $\dim_{\CC} M_0(\rho,L)$!}
\begin{cor} \label{c:dimM0}
 Let $V$ be a strongly regular VOA with exactly two isoclasses of irreducible $V$-modules and representatives $V$ and $M$ say, and assume that $\Gamma$ acts irreducibly on $\ch_V$. Let $c$ be the central charge of
 $V$, let $h$ be the conformal weight of $M$, and assume that $V$ is extremal in the sense above. Then
 \[
 \dim_{\CC} M_0 = \begin{cases}
 1728^hX, & k_1 = 0 \textrm{ or } -4,\\
 \tfrac{1+6h}{1-6h}1728^hX,& k_1 = -2,
\end{cases}
\]
where
\begin{gather*}
 X = 4^{-h}\frac{\Gamma(-h)\Gamma\big(\tfrac{5}{6}+h\big)}{\Gamma(h)\Gamma\big(\tfrac{5}{6}-h\big)}\sqrt{\frac{\sin\big(\pi\big(h-\tfrac{1}{6}\big)\big)}{\sin\big(\pi\big(h+\tfrac{1}{6}\big)\big)}}
\end{gather*}
and the sign of $X$ is chosen to ensure that $\dim_{\CC} M_0$ is positive. If $h > -\tfrac{5}{6}$ then this can be rewritten as
\[
 X = 4^{-h}{}_2F_1(-2h,-5/6;-h;1)\sqrt{\frac{\sin(\pi(h-\tfrac{1}{6}))}{\sin(\pi(h+\tfrac{1}{6}))}}.
\]
\end{cor}
\begin{proof} This follows immediately from Theorem~\ref{t:rank2} and Table~\ref{tablerank2}. The $X$ term is simplified as
\begin{align*}
 X &= \frac{\Gamma(-h)\Gamma\big(\tfrac{11}{12}+\tfrac{h}{2}\big)\Gamma\big(\tfrac{5}{12}+\tfrac{h}{2}\big)}{\Gamma(h)\Gamma\big(\tfrac{11}{12}-\tfrac{h}{2}\big) \Gamma\big(\tfrac{5}{12}-\tfrac{h}{2}\big)}\sqrt{-\frac{\sin\big(\pi\big(\tfrac{1}{12}-\tfrac{h}{2}\big)\big)\sin\big(\pi(\tfrac{7}{12}-\tfrac{h}{2})\big)} {\sin\big(\pi\big(\tfrac{1}{12}+\tfrac{h}{2}\big)\big)\sin\big(\pi\big(\tfrac{7}{12}+\tfrac{h}{2}\big)\big)}}\\
 &=4^{-h}\frac{\Gamma(-h)\Gamma\big(\tfrac{5}{6}+h\big)}{\Gamma(h)\Gamma\big(\tfrac{5}{6}-h\big)}\sqrt{-\frac{\sin\big(\pi\big(\tfrac{1}{12}-\tfrac{h}{2}\big)\big) \sin\big(\pi\big(\tfrac{1}{12}-\tfrac{h}{2}\big)+\tfrac{\pi}{2})}{\sin\big(\pi\big(\tfrac{1}{12}+\tfrac{h}{2}\big)\big)\sin\big(\pi\big(\tfrac{1}{12}+\tfrac{h}{2}\big)+\tfrac{\pi}{2}\big)}},
\end{align*}
where in the second line we have made use of Euler's relfection formula for $\Gamma$, and in the third we have used Legendre's duplication formula. Since $\sin(z + \tfrac{\pi}{2}) = \cos(z)$, the trigonometric factor then simplifies as claimed. Finally, when $h > -\tfrac{5}{6}$, the remaining $\Gamma$-factors can be reexpressed using Gauss's special value identity
\[
_2F_1(\alpha,\beta;\gamma;1) = \frac{\Gamma(\gamma)\Gamma(\gamma-\alpha-\beta)}{\Gamma(\gamma-\alpha)\Gamma(\gamma-\beta)},
\]
which is valid whenever $\Re(\gamma) > \Re(\alpha+\beta)$. Taking $\alpha = -2h$, $\beta = -5/6$, $\gamma = -h$ yields the last result.
\end{proof}

\begin{ex}
 In \cite{TG} all extremal VOAs with two simple modules were computed (see also~\cite{MNS}) and in all cases one has $h > -5/6$. Consider for example the previously unknown character vector of \cite{TG} corresponding to the values $c=33$, $h = \tfrac{9}{4}$ and $k_1 = -4$. Using a different method, the authors of \cite{TG} found that $\dim M_0 = 565760$. Using the material discussed above, this computation corresponds to the identity
 \[
-2^{9}3^{\tfrac{27}{4}}{}_2F_1(-9/2,-5/6;-9/4;1)\sqrt{2-\sqrt{3}} = 565760.
\]
Note that we needed to take the negative squareroot in this computation. After some manipulations this can be shown to be equivalent to the simpler identity
\[
\frac{\Gamma(3/4)\Gamma(5/12)}{\Gamma(1/4)\Gamma(11/12)}=\sqrt{2\sqrt{3}-3},
\]
which can be deduced from first principles using the techniques found in \cite{Vidunas}.
\end{ex}

\begin{ex}
 Corollary \ref{c:dimM0} can also be used to eliminate some possibilities from \cite{MNS}. In this paper, some cases with $c = -6, -8$ and $-10$ and $k_1=0$ proved to be particularly awkward to eliminate. In these cases we have $h = -\tfrac{1}{3},-\tfrac{1}{2},-\tfrac{2}{3}$ respectively. Using Corollary~\ref{c:dimM0} one can evaluate the formula for $\dim_{\CC}M_0$ at these points, and the values turn out to be nonintegral in these cases. Thus, there can be no corresponding strongly regular VOA realizing the corresponding character vectors. Note that these cases were already eliminated in \cite{TG} by a recursive computation, rather than by using exact formulas as above.
\end{ex}

The next most natural cases to consider in the classification of strongly regular VOAs with two simple modules are the cases when $k_1 = -6,-8$ or $-10$. In all of these cases the character vector lives in a two-dimensional vector space. For example, when $k_1 = -6$ there is a formula of the form $F_V = \alpha E_6F+\beta E_4DF$ where
\[
F = \eta^{-6}\left(\begin{matrix}{j^{\tfrac{h}{2}-1} {}_2F_1\big(1-\tfrac h2,\tfrac 43-\tfrac{h}{2}; 1-h;\tfrac{1728}{j}\big)}\\{1728^{h}Xj^{-\tfrac{h}{2}-1}{}_2F_1\big(1+\tfrac{h}{2},\tfrac 43+\tfrac{h}{2}; 1+h;\tfrac{1728}{j}\big)}\end{matrix}\right)
\]
and $\alpha$, $\beta$ are unknown scalars. Since the $q$-expansion of the first coordinate must have coefficient equal to $1$, one can write $\beta$ as a linear function of $\alpha$. Thus, $F_V$ lives in a two-parameter family determined by the invariants $(h,\alpha)$. Varying $h$ changes the monodromy, while varying~$\alpha$ corresponds to an isomonodromic deformation~\cite{Sabbah}. The reason that the extremal cases are so straightforward is due to the fact that one does not have to worry about isomonodromic deformations of the character vector.

We expect that the problem of classifying strongly regular VOAs with two modules such that $-10 \leq k_1 \leq -6$ will be closer in difficulty to the problem of classifying the Schellekens list of holomorphic VOAs with central charge $c = 24$, rather than the easier classification of the extremal cases discussed in~\cite{TG}. For example, suppose that one found parameters so that~$\alpha E_6F$ corresponds to a strongly regular VOA with rapidly growing graded pieces, while $\beta E_4DF$ corresponds to a discrete series Virasoro VOA with the same underlying monodromy representation. The coefficients of the character of such a Virasoro VOA are relatively slowly growing (see Example~\ref{ex:hardhexagon}), so that one could reasonably expect to find that all of the linear combinations
\[
H_m = m\alpha E_6F + (1-m)\beta E_4DF
\]
for integers $m\geq 1$ have positive integer Fourier coefficients. This would be an isomonodromic family of conformal modular forms similar to the Schellekens list family $j+m$, and we expect similarly that most of the specializations $H_m$ would not be realized by strongly regular VOAs. More generally, as $\dim M_0(\rho,L)$ grows, the corresponding classification problem for character vectors in $M_0(\rho,L)$ will obviously become more intractable.

\section{Higher ranks}\label{s:highrank}

\subsection{Rank 3}\label{s:rank3}
Hypergeometric series can be used to describe character vectors of strongly regular VOAs with two simple modules and, as suggested in~\cite{MasonFive}, one can use the results of~\cite{FM2} to tell a similar story for VOAs with three simple modules. In this case there is a free basis for the relevant module of modular forms $M(\rho,L)$ of the form $\big\{F, DF,D^2F\big\}$, where the minimal weight form $F$ can be described explicitly in terms of generalized hypergeometric series~${}_3F_2$. The weight of $F$ is $k_1 = 4\Tr(L)-3$. The \emph{extremal} cases where $M_0(\rho,L)$ are one-dimensional are given by $k_1=0$ and $k_1=-2$.

In \cite{FM6} we classified character vectors, and the corresponding strongly regular VOAs, for such~$\rho$ in cases where $k_1 = 0$. This means that $F$ satisfies a modular linear differential equation of weight $0$ of the form
\[
D^3F + aE_4D^2F+bE_6F = 0,
\]
an equation that becomes generalized hypergeometric when expressed on the $K$-line \cite{FM2}.

To classify the conformal modular forms arising from the corresponding Frobenius family, we examined the first nontrivial Fourier coefficient of $F$ expressed as a vector of algebraic functions of the conformal weights $h_1$ and $h_2$. Since these coefficients are obtained from the Frobenius method, their divisors turn out to have many linear factors coming from their denominators, cf.\ Lemma~\ref{l:frobfam}(b). Since we only needed to consider series with positive Fourier coefficients, this allowed us to reduce the search for conformal specializations of the Frobenius family $F$ to two relatively small regions of the $(h_1,h_2)$-plane bounded by two horizontal and two diagonal lines occuring in the divisors. Once this reduction was performed, we then used results on the arithmetic of $_3F_2$ to determine exactly what specializations of the Frobenius family~$F$ had a~first coordinate with nonnegative integer coefficients. This relatively minimal amount of input already reduced us to considering one infinite family of examples, plus a finite list of additional possible conformal specializations arising from the Frobenius family. The last step in our classification strategy was to compute and symmetrize a finite number of $S$-matrices, and then check which of the integral specializations remained integral after changing basis to make~$\rho(S)$ symmetric.

The result of all this was an infinite family of conformal modular forms, in addition to a finite list of additional examples. The data is contained in Figure~\ref{f:rank3}, where we plot the examples as points in the $(h_1,h_2)$-plane. We showed that the infinite family was realized by~$B_{\ell,1}$, and the exceptional examples contained~$A_{1,2}$, $\Vir(c_{3,4})$ and~$\Vir(c_{2,7})$. In addition to these known examples, we found~$11$ more conformal modular forms that seemed to live in a family containing~$E_{8,2}$ and the Baby Monster theory, but such that the $9$ other entries in the family did not correspond to known examples. In Figure~\ref{f:rank3} this new family corresponds to the blue dots at height~$3/2$. We called this family the $U$-series. In \cite[Section~9]{FM6} we proved that all but three of these exceptions can be realized as commutants inside the Moonshine module $V^{\natural}$.

\begin{figure}[t]\centering
 \includegraphics[scale=0.7]{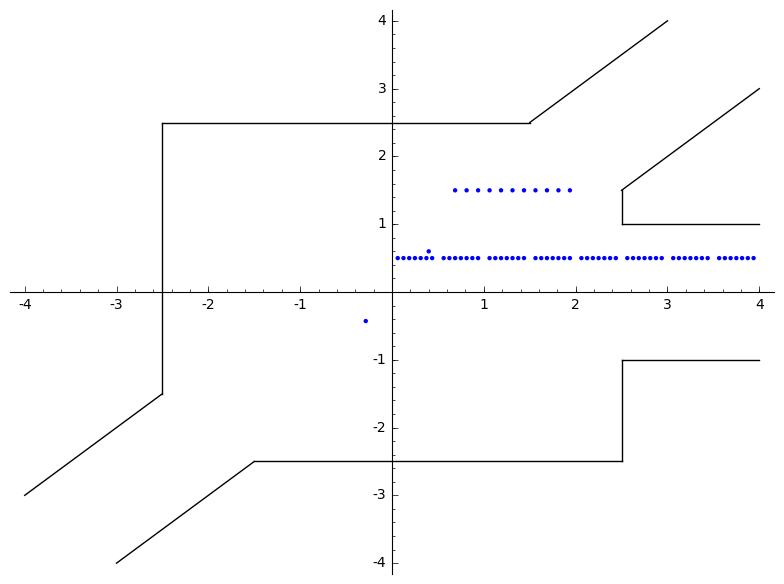}
 \caption{Conformal modular forms in rank $3$ in the extremal case when $k_1 = 0$. The dots at height $h_2 = 1/2$ correspond to $A_{1,2}$ and $B_{\ell,1}$. The dots at height $h_2 = 3/2$ are the $U$-series. The exceptional dots are $\Vir(c_{2,7})$ and $\Vir(c_{3,4})$. The black boundary lines correspond to the positivity conditions arising from the analysis of the first nontrivial $q$-series coefficient of the Frobenius family.} \label{f:rank3}
\end{figure}

\subsection{Rank 4 and higher}\label{s:rank4}
Once one begins to consider VOAs with four or more simple modules, the discussion of characters becomes more complicated. One reason for this is that there are two possibilities for the structure of the module of modular forms associated to an irreducible representation $\rho \colon \Gamma \to \GL_4(\CC)$ and choice of exponents~$L$~-- see~\cite{FM5} for a detailed discussion. The \emph{cyclic case} corresponds to~$M(\rho,L)$ being generated as a $D$-module by a single form $F$ of minimal weight $k_1 = 3\Tr(L)-3$, so that a~free-basis over $M$ is $F$, $DF$, $D^2F$ and $D^3F$. The \emph{noncyclic case} corresponds to modules $M(\rho,L)$ with one generator in weight $k_1 = 3\Tr(L)-2$, two in weight $k_1+2$, and one in weight $k_1+4$. In~\cite{FM5} we studied vvmfs associated to tensor products, symmetric powers and inductions of two-dimensional representations. One could analyze the corresponding Frobenius families in an attempt to classify potential characters of VOAs. In this section we will carry this out in one of the simplest examples in rank~$4$, by inducing from a~subgroup of index $4$. See~\cite{Bae} for a~similar computation with $\Gamma_0(2)$ instead of~$\Gamma_0(3)$.

While $\Gamma$ has three distinct conjugacy classes of subgroups of index $4$, only the conjugacy class of $\Gamma_0(3)$ has infinitely many one-dimensional representations that give rise to a nonconstant Frobenius family. We study this family next.
\begin{rmk}The cases of $\Gamma_0(p)$ for $p=5$ and $13$ could be treated in a similar fashion, yielding families of rank $6$ and $14$, respectively.
\end{rmk}

\subsubsection[A family associated to Gamma\_0(3)]{A family associated to $\boldsymbol{\Gamma_0(3)}$}

We begin by classifying the irreducible congruence representations of $\Gamma_0(3)$. Let $U = S^{-1}T^{-1}S$ and set $A=U^3T^{-1}$, so that $\Gamma_0(3)=\big\langle T, U^3\big\rangle=\langle T, A\rangle = \big\langle U^3, A\big\rangle=\Gamma(3)\langle A\rangle$. Consider the subgroup $\big\langle U^3, AU^3A^{-1}, A^2U^3A^{-2}\big\rangle$, which is normalized by $A$ and contained in $\Gamma(3)$. It follows easily that in fact
\[\Gamma(3)=\big\langle U^3, AU^3A^{-1}, A^2U^3A^{-2}\big\rangle.\]
Let $\bar{\Gamma}(3)$ denote the image of $\Gamma(3)$ in $\PSL_2(\ZZ)$. Then the group $\bar{\Gamma}(3)$ is free of rank $3$, and therefore it is freely generated by the~$3$ indicated elements (after barring them). Now $\bar{\Gamma}(3)/\bar{\Gamma}(3)'\cong \ZZ^3$. Putting these facts together shows that $\bar{\Gamma}_0(3)/\bar{\Gamma}(3)' \cong \ZZ \wr\ZZ_3$.

\begin{lem}\label{lemcomm} We have
\begin{gather*}
\bar{\Gamma}_0(3)/\bar{\Gamma}_0(3)' \cong\ZZ\times \ZZ_3
\end{gather*}
and the two factors in the quotient are generated by the images of $U^3$ and $A$ respectively.
\end{lem}
\begin{proof}
The group $\ZZ \wr C_3$ may be realized as $\ZZ^3$ extended by an order $3$ automorphism that cyclically permutes coordinates. Then we calculate that the commutator subgroup consists of $(a, b, c)\in\ZZ^3$ with coordinate sum $a+b+c=0$. Therefore the commutator quotient of $\ZZ \wr\ZZ_3$ is isomorphic to $\ZZ\times \ZZ_3$. The proof follows.
\end{proof}

After Lemma \ref{lemcomm}, the characters $\chi\colon \bar{\Gamma}_0(3)\longrightarrow \CC^\times$ of \emph{finite order} are indexed by pairs $(\lambda, \epsilon)$ where $\lambda$ is a root of unity, and $\epsilon\in\{0, \pm 1\}$, and we have
\begin{gather*}
\chi(U^3) =\lambda, \qquad \chi(A) =\omega^{\epsilon}.
\end{gather*}
For any of the characters $\chi=\chi_{\lambda, \epsilon}$ described in \eqref{chidef} we set
\[
\rho\df \Ind_{\bar{\Gamma}_0(3)}^{\bar{\Gamma}} \chi.
\]
These are the $4$-dimensional monomial representations of interest. We need a good left transversal for $\bar{\Gamma}_0(3)$ in~$\bar{\Gamma}$. We take left coset representatives to be $I$, $\bar{U}$, $\bar{U}^2$, $\bar{S}$, and choose an ordered basis for the linear space that furnishes $\rho$ to correspond to the left cosets of $\bar{\Gamma}_0(3)$ where we use these group elements in the indicated order. Name this ordered basis $(v_1, v_2, v_3, v_4)$. Then we obtain
\begin{gather*}
\rho(\bar{U})=\left(\begin{matrix}0 & 0 & \lambda & 0 \\ 1 & 0 & 0 & 0 \\0 & 1 & 0 & 0 \\0 & 0 & 0 & \bar{\lambda}\omega^{\epsilon} \end{matrix}\right),\qquad
\rho(\bar{S})=\left(\begin{matrix}0 & 0 & 0 & 1 \\0 & 0 & \omega^{\epsilon} & 0 \\0 & \bar{\omega}^{\epsilon} & 0 & 0 \\ 1 & 0 & 0 & 0\end{matrix}\right),\\
\rho(\bar{A})= \left(\begin{matrix}\omega^{\epsilon} & 0 & 0 & 0 \\0 & 0 & \lambda\bar{\omega}^{\epsilon} & 0 \\0 & 0 & 0 & \lambda\bar{\omega}^{\epsilon} \\
0 & \bar{\lambda}^2\bar{\omega}^{\epsilon} & 0 & 0\end{matrix}\right).
\end{gather*}
We find that
\begin{gather*}
\rho\big(\bar{U}^3\big)= \diagg\big(\lambda, \lambda, \lambda, \bar{\lambda}^3\big), \qquad
\rho(\bar{T}) = \rho\big(\bar{A}^{-1}\bar{U}^3\big)=\left(\begin{matrix}\lambda\bar{\omega}^{\epsilon} & 0 & 0 & 0 \\0 & 0 & 0 & \bar{\lambda}\omega^{\epsilon} \\ 0 & \omega^{\epsilon} & 0 & 0 \\0 & 0 & \omega^{\epsilon} & 0
\end{matrix}\right).
\end{gather*}

\begin{lem}\label{lemsolvable} $\im\rho$ is described by a short exact sequence
\begin{gather*}
1\rightarrow K \rightarrow \im\rho\rightarrow A_4 \rightarrow 1,
\end{gather*}
where $K\df \rho(\bar{\Gamma}(3))$ is an abelian normal subgroup.
\end{lem}
\begin{proof} The group $\bar{\Gamma}(3)$ leaves the span of the first coordinate vector invariant. Since $\bar{\Gamma}$ acts transitively on the basis then $\bar{\Gamma}(3)$ acts as a group of \textit{diagonal} matrices, hence it furnishes an abelian, normal subgroup of $\im\rho$, call it $K$. Now it suffices to show that the quotient group by~$K$ is isomorphic to~$A_4$. Indeed, $\bar{\Gamma}/\bar{\Gamma}(3)\cong A_4$, and the nature of the displayed matrices~$\rho\big(\bar{U}\big)$,~$\rho\big(\bar{A}\big)$, both of which project onto elements of order~$3$, shows that $\im\rho/K\cong A_4$. The Lemma is proved.
\end{proof}

\begin{lem} \label{l:ind4irr}
 $\rho$ is irreducible if, and only if, $\lambda^4\neq 1$.
\end{lem}
\begin{proof} First assume that $\lambda^4\neq 1$, and take an invariant subspace $X$ of the module furnishing~$\rho$. $X$ is spanned by eigenvectors for $U^3$. Now, $X$ cannot contain $v_4$ because this generates the full module. Therefore, $X\subseteq\langle v_1, v_2, v_3\rangle$. In fact, we claim that $X\subseteq\langle v_2, v_3\rangle$. Otherwise, some vector in~$X$ projects nontrivially onto $v_1$ and then application of $S$ shows that some vector in~$X$ projects nontrivially onto $v_4$, an impossibility. Now $U\colon v_2\mapsto v_3 \mapsto \lambda v_1$, and we find that the only possibility is $X=0$.

Now suppose that $\lambda^4=1$. Then $U^3$ is represented by the scalar matrix $\lambda I_4$, whence all elements of $\pm \Gamma(3)$ are represented as scalars, because this group is the normal closure of $\pm U^3$ in~$\Gamma$. Then $\im\rho$ is isomorphic to a central extension (by the group $\langle \lambda \rangle$) of a quotient group of~$A_4$. This quotient group must be one of $A_4$, $C_3$, $1$, and no central extension of any of these groups has a $4$-dimensional irreducible representation. In other words, $\rho$ cannot be irreducible if $\lambda^4=1$. This completes the proof of the lemma.
\end{proof}

\begin{lem}\label{lemN} Suppose that $\ker\rho$ contains the principal congruence subgroup $\bar{\Gamma}(N)$, and suppose that $N$ is the \textit{least} such integer (the \textit{level} of $\ker \rho$). Then $N=2^a3^b$ for nonnegative integers~$a$,~$b$.
\end{lem}
\begin{proof}We have $\im\rho\cong \Gamma/\ker \rho \cong(\Gamma/\Gamma(N)/(\ker \rho/\Gamma(N))$. Suppose that $N=Mp^k$ for some prime $p\geq 5$ where $M$ is coprime to $p$ and $k\geq 1$. Then $\Gamma/\Gamma(N)$ has a direct factor isomorphic to~$\Gamma/\Gamma\big(p^k\big)$. This last group is \textit{perfect} (i.e., it coincides with its commutator subgroup), and hence it has \textit{no} nontrivial solvable quotient groups and in particular $\im\rho$ must be nonsolvable. This contradicts Lemma~\ref{lemsolvable}.
\end{proof}

\begin{ex}Suppose that $\lambda=1$. In this case we have seen that $\ker\rho=\bar{\Gamma}(3)$ and $K=1$. In effect, then, $\rho$ is the representation of $\bar{\Gamma}/\bar{\Gamma}(3)\cong A_4$ that obtains if we induce the $1$-dimensional character
of a Sylow $3$-subgroup $A_4$ to the full group.
\end{ex}

\begin{ex}
We first recall that for each $k\geq 1$ there is a short exact sequence
\begin{gather*}
1\rightarrow\bar{\Gamma}(3)/\bar{\Gamma}\big(3^k\big)\rightarrow\bar{\Gamma}/\bar{\Gamma}\big(3^k\big) \rightarrow A_4\rightarrow 1
\end{gather*}
and that furthermore $P\df \Gamma(3)/\Gamma(3^k)$ is a finite $3$-group. Then a $3$-Sylow-subgroup of the middle group $\Gamma/\Gamma\big(3^k\big)$ may be taken to be $P\rtimes\langle A\rangle$. Take a $1$-dimensional character $\chi\colon P\rtimes\langle A\rangle\rightarrow \CC^\times$ and let $\chi_0$ be the restriction of $\chi$ to $P$. This character is $\langle A\rangle$-invariant, and we assume that the \textit{stabilizer} of $\psi_0$ in $A_4$ is nothing but $\langle A\rangle$.
In this case we know by Clifford theory, cf.~\cite[Section~11B]{CurtisReiner}, that the induced representation $\rho\df \Ind_{P\rtimes\langle A\rangle}^{\bar{\Gamma}/\bar{\Gamma}(3^k)}$ is irreducible.
\end{ex}

\begin{ex}
We have $\abs{\bar{\Gamma}(3)/\bar{\Gamma}(9)}=3^3$. Now calculate that this group is abelian of exponent~$3$, so that $\bar{\Gamma}(3)/\bar{\Gamma}(9)\cong \ZZ_3^3$. Then we can get several admissible characters $\chi$ with $\lambda$ a~primitive cube root of unity and any $\epsilon$.
\end{ex}

\begin{lem} For all integers $f\geq 2$ and $M\geq 1$ coprime to $3$, we have $\bar{\Gamma}\big(3^fM\big)\bar{\Gamma}(3)'=\bar{\Gamma}(9)$.
\end{lem}
\begin{proof}
 If $f=2$, $M=1$ this follows from the preceding Example $3$, so we may suppose that $f\geq 3$. By direct calculation we can find a matrix $M\in \bar{\Gamma}(3)'\setminus{\bar{\Gamma}\big(3^3\big)}$. Since $1\neq M\bar{\Gamma}\big(3^3\big)\subseteq\bar{\Gamma}(9)$ it follows that $\bar{\Gamma}\big(3^{f}\big)\bar{\Gamma}(3)'$ projects \textit{onto} the quotient $\bar{\Gamma}(9)/\bar{\Gamma}\big(3^{f}\big)$. This proves the lemma if $M=1$.

If $M>1$ then $\bar{\Gamma}/\bar{\Gamma}\big(3^fM\big)\cong \bar{\Gamma}/\bar{\Gamma}(M) \times \bar{\Gamma}/\bar{\Gamma}\big(3^f\big)$. Now apply the previous argument.
\end{proof}

\begin{cor}
 In the notation of Lemma~{\rm \ref{lemN}}, we always have $b\leq 2$.
\end{cor}

\begin{lem}
 In the notation of Lemma~{\rm \ref{lemN}}, we always have $a\leq 2$.
\end{lem}
\begin{proof} Let the notation be as before and let $K_0$ be the \textit{odd} part of $K$, that is, the subgroup generated by all elements of odd order. Then $K/K_0$ is either cyclic of order $2^x$ or homocyclic~$(C_{2^x})^3$. In the first case our $A_4$ group must act trivially on $K/K_0$, and this implies that $\im\rho$ has a quotient isomorphic to $C_{2^{x-1}}$. Therefore $x\leq 2$ in this case. In the homocyclic case, we observe that the $A_4$ subgroup leaves invariant a $(C_{2^x})^2$-subgroup of $(C_{2^x})^3$, and therefore we may apply the previous argument to the $C_{2^x}$-quotient group to again see that $x\leq 2$. This completes the proof of the lemma.
\end{proof}

This finally brings us to
\begin{thm}
 If the induced representation $\rho$ is a congruence representation then its level must divide $36$.
\end{thm}

Now we would like to describe a family of vvmfs corresponding to this data. To this end, notice that $\Gamma_0(3)$ corresponds to the dessins d'enfants
\begin{center}
\begin{tikzpicture}
\draw (1,0) circle (1cm);
\draw (0,0) circle (0.15 cm);
\filldraw[black] (2,0) circle (0.15 cm);
\draw[thick] (2,0)--(4,0);
\draw (3,0) circle (0.15 cm);
\filldraw[black] (4,0) circle (0.15cm);
\end{tikzpicture}
\end{center}
whose associated Belyi function $z$ over the $j$-line~-- ramified at $0$, $1728$ and $\infty$~-- is defined by the minimal polynomial:
\[
 j = \frac{(z+27)(z+243)^3}{z^3}.
\]
The roots of this polynomial generate the function fields of the conjugates of $\Gamma_0(3)$. Newton's method yields expressions for these conjugates in terms of $1/j$: if $x=j^{-1/3}$ then the conjugates have expansions
\begin{gather*}
 z_1 = j-756-196830j^{-1}-167935356j^{-2}-180552296781 j^{-3}-\cdots,\\
 z_2 = 729\big(x+12x^2+90x^3+756x^4+8343x^5+76788x^6+729000x^7+835340x^8+\cdots\big),
\end{gather*}
$z_3 = z_2(ux)$ and $z_4 = z_2\big(u^2x\big)$ where $u^2+u+1=0$. This leads to the following $q$-series expansions:
\begin{gather*}
z_1= \frac{1}{q}-12+54q-76q^2-243q^3+1188q^4-1384q^5-2916q^6+11934q^7+\cdots\\
 z_2 = 729\big(q^{1/3}+12q^{2/3}+90q+508q^{4/3}+2391q^{5/3}+9828 q^{2}+36428q^{7/3}+\cdots\big),
\end{gather*}
$z_3 = z_2\big(uq^{1/3}\big)$ and $z_4 = z_2\big(u^2q^{1/3}\big)$.

It is well-known that $z_1 = \big(\eta(q)/\eta\big(q^3\big)\big)^{12}$ and $z_2 = 729\big(\eta(q)/\eta\big(q^{1/3}\big)\big)^{12}$. Notice that if we choose $u = {\rm e}^{2\pi{\rm i}/3}$, then $T$ fixes $z_1$ and permutes the other $z_j$ as $z_2\mapsto z_3\mapsto z_4\mapsto z_2$. From the transformation law of $\eta$ one finds that $z_1|S=z_2$. Since $S$ acts on these functions as a~permutation of order $2$, and it does not fix any of them (otherwise the group leaving them invariant would be too large) we see that $T$ and $S$ act on the vvmf $(z_1,z_2,z_3,z_4)^{\rm T}$ via the matrices
\begin{gather*}
\rho'(T) = \left(\begin{matrix}
 1&0&0&0\\
 0&0&1&0\\
 0&0&0&1\\
 0&1&0&0
 \end{matrix}\right),\qquad \rho'(S) = \left(\begin{matrix}
 0&1&0&0\\
 1&0&0&0\\
 0&0&0&1\\
 0&0&1&0
 \end{matrix}\right).
\end{gather*}

The modular derivative $D$ in weight zero is $D=q{\rm d}/{\rm d}q$. Then
\[
 f_1\df -\frac{D(z_1)}{z_1} = 1+12q+36q^2+12q^3+84q^4 +\cdots
\]
is the unique eigenform of weight $2$ for $\Gamma_0(3)$. Said differently, $z_1$ is the unique solution of the MLDE $(D+f_1)z_1=0$. More generally, let $f_j=-D(z_j)/z_j$ and consider the MLDEs $(D+\lambda f_j)g=0$ which have the solutions $g=z_j^\lambda$. This $\lambda$ is the exponent deformation parameter of our family.

Our next goal is to find the induced MLDE of rank $4$ for the full modular group. To this end, recall that we have $D(E_4) = -(1/3)E_6$, $D(E_6) = -(1/2)E_4^2$ and $D(\Delta) = 0$. Thus,
\[
 D(j) = \frac{-E_4^2E_6}{\Delta} = -\frac{E_6j}{E_4}.
\]
Differentiating the defining Belyi relation then gives
\[
 f_j = \frac{E_6(z_j+27)(z_j+243)}{E_4(z^2_j-486z_j-19683)}.
\]
It follows that $D(f_j) = -\tfrac 14 f_j^2 + \tfrac{1}{12}E_4$ and thus
\begin{gather*}
 D\big(z_j^\lambda\big) = -\lambda f_jz_j^\lambda,\\
 D^2\big(z_j^\lambda\big)= \left(\tfrac{4\lambda^2+\lambda}{4} f_j^2 - \tfrac{\lambda}{12}E_4\right)z_j^\lambda,\\
D^3\big(z_j^\lambda\big) = \left(-\tfrac{8\lambda^3+6\lambda^2+\lambda}{8} f_j^3+\tfrac{6\lambda^2+\lambda}{24} f_jE_4 + \tfrac{\lambda}{36}E_6\right) z_j^\lambda,\\
D^4\big(z_j^\lambda\big) = \left(\tfrac{32\lambda^4+48\lambda^3+22\lambda^2+3\lambda}{32} f_j^4 -\tfrac{12\lambda^3+7\lambda^2+\lambda}{24} f_j^2E_4 + \tfrac{2\lambda^2-\lambda}{96}E_4^2-\tfrac{8\lambda^2+\lambda}{72} f_jE_6\right) z_j^\lambda.
\end{gather*}
Note that $f_j^4$, $f_j^2E_4$, $E_4^2$, $f_jE_6$ are linearly dependent:
\[
 f_j^4 = \tfrac 23 f_j^2E_4+\tfrac{1}{27}E_4^2+\tfrac{8}{27}f_jE_6.
\]
Therefore the relation for $D^4(z_j^\lambda)$ above can be reexpressed as
\[
D^4\big(z_j^\lambda\big) = \left(\tfrac{32\lambda^4+24\lambda^3+8\lambda^2+\lambda}{48} f_j^2E_4 +\tfrac{16\lambda^4+24\lambda^3+20\lambda^2-3\lambda}{432}E_4^2+\tfrac{64\lambda^4+96\lambda^3+20\lambda^2+3\lambda}{216} f_jE_6\right) z_j^\lambda.
\]
If $z_j^\lambda$ satisfies a monic MLDE of level one, then we should be able to find a relation:
\[
 D^4\big(z_j^\lambda\big) +A(\lambda) E_4D^2\big(z_j^\lambda\big)+B(\lambda) E_6D\big(z_j^\lambda\big)+C(\lambda) E_4^2z_j^\lambda = 0
\]
for rational functions $A,B,C \in \QQ(\lambda)$. Since $f_j^2E_4$, $E_4^2$ and $f_jE_6$ are linearly independent, one deduces that
\begin{gather*}
 A(\lambda) = \tfrac 23 \lambda^2+\tfrac 13\lambda+\tfrac{1}{12},\\
 B(\lambda) = \tfrac{1}{9}\lambda^3+\tfrac{4}{9}\lambda^2+\tfrac{5}{54}\lambda+\tfrac{1}{72},\\
 C(\lambda) = -\tfrac{1}{27}\lambda\big(\lambda-\tfrac 12\big)\big(\lambda^2+\tfrac 12\lambda+\tfrac 34\big).
\end{gather*}
This shows that we are in the cyclic case of \cite{FM5}.

Our goal now is to find the monodromy for $\big(z_1^\lambda,z_2^\lambda,z_3^\lambda,z_4^\lambda\big)^{\rm T}$, which abstractly is isomorphic with the induced representations $\rho$ discussed previously, with $\veps = 1$. But since we will want to symmetrize~$\rho'(S)$, we will need to determine this representation exactly within its isomorphism class.

Since $T$ and $S$ permute the differential equations $D+\lambda f_j$ in the same way as they do the $z_j$, we find that there are functions of $\lambda$ such that $\Gamma$ acts on the solutions $z_j^\lambda$ as
\begin{align*}
\rho'(T) &= \left(\begin{matrix}
 A(\lambda)&0&0&0\\
 0&0&B(\lambda)&0\\
 0&0&0&C(\lambda)\\
 0&D(\lambda)&0&0
 \end{matrix}\right),& \rho'(S) &= \left(\begin{matrix}
 0&E(\lambda)&0&0\\
 F(\lambda)&0&0&0\\
 0&0&0&G(\lambda)\\
 0&0&H(\lambda)&0
 \end{matrix}\right).
\end{align*}
The identity $\rho'(S)^2 =1$ implies that $F=E^{-1}$ and $H=G^{-1}$. Then the condition that $\rho'(R)^3 =1$ implies that $ABDG=1$ and $C^3=H^3$. The monodromy $\rho'(T)$ can in fact be read off from our expansions:
\begin{gather*}
 z_1(\tau)^\lambda ={\rm e}^{-2\pi{\rm i}\lambda\tau}\prod_{n\geq 0}\left(\frac{1-q^{n}}{1-q^{3n}}\right)^{12\lambda},\\
 z_2(\tau)^\lambda =729^\lambda {\rm e}^{2\pi{\rm i} \lambda \tau/3}\prod_{n\geq 0}\big(1+{\rm e}^{2\pi{\rm i} n\tau/3}+{\rm e}^{4\pi{\rm i} n\tau /3}\big)^{12\lambda},\\
 z_3(\tau)^\lambda =729^\lambda {\rm e}^{2\pi{\rm i} \lambda(\tau+1)/3}\prod_{n\geq 0}\big(1+u{\rm e}^{2\pi{\rm i} n\tau/3}+u^2{\rm e}^{4\pi{\rm i} n\tau /3}\big)^{12\lambda},\\
 z_4(\tau)^\lambda =729^\lambda {\rm e}^{2\pi{\rm i} \lambda (\tau+2)/3}\prod_{n\geq 0}\big(1+u^2{\rm e}^{2\pi{\rm i} n\tau/3}+u{\rm e}^{4\pi{\rm i} n\tau /3}\big)^{12\lambda},
\end{gather*}
so that we have
\begin{gather*}
\rho'(T) = \left(\begin{matrix}
 {\rm e}^{-2\pi{\rm i} \lambda}&0&0&0\\
 0&0&1&0\\
 0&0&0&1\\
 0&{\rm e}^{2\pi{\rm i} \lambda}&0&0
 \end{matrix}\right),\qquad \rho'(S) = \left(\begin{matrix}
 0&E(\lambda)&0&0\\
 E(\lambda)^{-1}&0&0&0\\
 0&0&0&1\\
 0&0&1&0
 \end{matrix}\right).
\end{gather*}
The transformation law for $\eta$ under $S$ shows that likewise $E(\lambda)=1$. Diagonalizing $\rho'(T)$ and symmetrising $\rho'(S)$ shows the following:
\begin{thm} \label{t:inducedfamily} Let
 \[
 F = \frac{1}{\sqrt{3}}\left(\begin{matrix}
 \sqrt{3}z_1^\lambda \\ z_2^\lambda+u^{-\lambda}z_3^\lambda+u^{-2\lambda}z_4^\lambda\\ z_2^\lambda+u^{2-\lambda}z_3^\lambda+u^{1-2\lambda}z_4^\lambda\\ z_2^\lambda+u^{1-\lambda}z_3^\lambda+u^{2-2\lambda}z_4^\lambda
 \end{matrix}
 \right).
\]
Then $F$ is a vector-valued modular form of weight $0$ satisfying the transformation laws
\begin{gather*}
F(\tau+1) = \left(\begin{matrix}
 {\rm e}^{-2\pi{\rm i} \lambda} &0 &0 &0\\
 0 & {\rm e}^{2\pi{\rm i} \lambda/3} & 0 & 0\\
 0 & 0 & {\rm e}^{2\pi{\rm i} (\lambda+1)/3} & 0\\
 0 & 0 & 0 & {\rm e}^{2\pi{\rm i}(\lambda + 2)/3}
 \end{matrix}
 \right)F(\tau),\\
 F(-1/\tau) = \frac{1}{3}\left(\begin{matrix}
 0&\sqrt{3} &\sqrt{3} &\sqrt{3}\\
 \sqrt{3} &2\cos(2\pi \lambda/3)& 2\cos(2\pi (\lambda+2)/3)& 2\cos(2\pi (\lambda+1)/3)\\
 \sqrt{3} & 2\cos(2\pi (\lambda+2)/3)& 2\cos(2\pi (\lambda+1)/3)& 2\cos(2\pi \lambda/3)\\
 \sqrt{3} & 2\cos(2\pi (\lambda+1)/3) & 2\cos(2\pi \lambda/3) & 2\cos(2\pi (\lambda+2)/3)
 \end{matrix}
 \right)F(\tau).
\end{gather*}
In fact, $F$ is a Frobenius family defined in terms of the parameter $\lambda$.
\end{thm}
Let now $\rho$ denote the diagonalized representation as in Theorem~\ref{t:inducedfamily} above, and fix exponents $L = \diagg\big({-}\lambda,\tfrac{\lambda}{3},\tfrac{\lambda+1}{3},\tfrac{\lambda+2}{3}\big)$. If $\lambda$ is not a fourth root of unity then $\rho$ is irreducible by Lemma~\ref{l:ind4irr}, and \cite[Theorem~16]{FM5} shows that $F$ is a form of minimal weight in the cyclic $D$-module $M(\rho,L)$. That is, $F$, $DF$, $D^2F$ and $D^3F$ generate $M(\rho,L)$ as $M$-module.

In order to classify characters of VOAs arising in the family $F$ of Theorem~\ref{t:inducedfamily}, we'll need the first few terms in the $q$-expansion of $F$:
\[
 F = \left(\begin{matrix}
 q^{-\lambda}\big(1-12\lambda q + 72\lambda\big(\lambda - \tfrac 14\big) q^2-288\lambda \big(\lambda^2-\tfrac 34\lambda + \tfrac{1}{72}\big)q^3+\cdots\big)\\
 3^{6\lambda+\tfrac 12}q^{\lambda/3}\big(1+288\lambda\big(\lambda^2+\tfrac{3}{4}\lambda+\tfrac{1}{72}\big)q + \cdots\big)\\
 3^{6\lambda+\tfrac 12}q^{(\lambda+1)/3}\big(12\lambda+864\lambda\big(\lambda^3+\tfrac 32\lambda^2+\tfrac{35}{144}\lambda+\tfrac{7}{288}\big)q+\cdots\big)\\
 3^{6\lambda+\tfrac 12}q^{(\lambda+2)/3}\big(72\lambda\big(\lambda+\tfrac14\big)+\tfrac{10368}{5}\lambda(\lambda+2)\big(\lambda+\tfrac 14\big)\big(\lambda+\tfrac 16\big)\big(\lambda+\tfrac{1}{12}\big)q+\cdots\big)
 \end{matrix}\right).
\]
At this point, since we must preserve the facts that $\rho(T)$ is diagonal and~$\rho(S)$ is real symmetric, we have freedom to permute coordinates of~$F$, swaps signs, and rescale the entire vector by an overall factor. We require in particular that the coefficients of each coordinate are all of the same sign. By considering the $q$ coefficient of the first coordinate, this already forces us to consider $\lambda< 0$. By considering the $q$-coefficient of the second coordinate we find that we likewise require $\lambda > -3/4$.

Further, we need these expansions to be rational. This implies that we must have $6\lambda + \frac 12 \in \ZZ$, or that is, $12\lambda \in 1+2\ZZ$. Therefore, we only have to worry about the values $\lambda = \tfrac{n}{12}$ for odd $n$ in the range $-8\leq n \leq -1$. When $n=-3,-6$ the monodromy is reducible, so we in fact only have to consider the values $\lambda = -\tfrac{1}{12}$, $-\tfrac{5}{12}$ and $-\tfrac{7}{12}$. The cases $\lambda = -1/12$ and $-7/12$ are accounted for by $\Vir(c_{2,9})$ and $G_{2,2}$, respectively (after reordering and rescaling basis vectors appropriately). The case $\lambda = -5/12$ can be ruled out by noticing that some fusion rules for the corresponding $S$-matrix are $-1$. Thus, this specialization gives a quasi-conformal modular form that can't be re-ordered and scaled to give a conformal modular form. We summarize this as:
\begin{thm}
The only specializations of~$F$ that give rise to characters of strongly regular VOAs with irreducible monodromy and exactly $4$ simple modules $($possibly after rescaling and reordering coordinates of~$F)$ are the values $\lambda = -\tfrac{1}{12}$ and $-\tfrac{7}{12}$, which correspond to~$\Vir(c_{2,9})$ and~$G_{2,2}$, respectively.
\end{thm}

Once such a family $F$ is computed, it is possible to derive other familes from it in the search for new character vectors. For example, let now $G = \eta^{-4}D_0F$, so that the exponents have shifted by $-1/6$, and the $S$-matrix has been replaced by its negative. This is another extremal Frobenius family where the specializations for generic $\lambda$ live in a one-dimensional space of vector-valued modular forms (though they are no longer minimal weight, as the minimal weight has shifted to $k_1 = -2$).

We find that
\begin{gather*}
 G = \left(\begin{matrix}
 q^{-\lambda-\tfrac 16}\big({-}\lambda+12\lambda\big(\lambda -\tfrac 43\big)q-72\lambda\big(\lambda-\tfrac{7}{3}\big)\big(\lambda - \tfrac{7}{12}\big)q^2+\cdots\big)\\
 3^{6\lambda+\tfrac 12}q^{(2\lambda-1)/6}\big({-}\lambda -288\lambda(\lambda-1)\big(\lambda^3+\tfrac 34\lambda +\tfrac{1}{72}\big)q+ \cdots\big)\\
 3^{6\lambda+\tfrac 12}q^{(2\lambda+1)/6}\big({-}12\lambda^2-864\lambda(\lambda-1)\big(\lambda^3+\tfrac 32\lambda^2+\tfrac{35}{144}\lambda+\tfrac{7}{288}\big)q+\cdots\big)\\
 3^{6\lambda+\tfrac 12}q^{(2\lambda+3)/6}\big({-}72\lambda^2\big(\lambda\!+\!\tfrac 14\big)\!-\!\tfrac{10368}{5}\lambda(\lambda\!-\!1)\big(\lambda\!+\!\tfrac{1}{12}\big)\big(\lambda\!+\!\tfrac 16\big)\big(\lambda\!+\!\tfrac 14\big)(\lambda\!+\!2)q\!+\!\cdots\big)
 \end{matrix}\right)\!.
\end{gather*}
As above, rationality implies that we must have $12\lambda$ equal to an odd integer. Irreducibility implies that $12\lambda$ is coprime to $3$, so that $12\lambda \equiv 1,5\pmod{6}$. Next, the fact that all of the Fourier coefficients in each coordinate of $G$ should have the same sign implies that if $\lambda > 0$ then $\lambda\leq 7/12$, as one sees by examining the first coordinate. If $\lambda < 0$ then examination of the second coordinate implies that one should have $\lambda > -\tfrac{9}{12}$. That is, we are only interested in rational $\lambda$ satisfying $\lambda = \tfrac{n}{12}$ where $-7 \leq n \leq 7$ and $n\equiv 1,5\pmod{6}$. Consideration of integrality and fusion rules then implies that only the case $\lambda=-1/12$ could plausibly correspond to a VOA.

In order to correct fusion rules and signs, one needs to adjust the specialization of $G$ at $\lambda=-1/12$ by permuting and rescaling as follows:
\begin{equation}\label{eq:H}
 H = \left(\begin{matrix}
q^{\tfrac{17}{36}}\big(1 + 25q^2+ 133q^3 +578q^4+ 1970q^5+ 6076q^6+16840q^7+\cdots\big)\\
 q^{-\tfrac{7}{36}}\big(1+13q+ 98q^2+471q^3+1780q^4+5765q^5+ 16856q^6+\cdots\big)\\
 q^{\tfrac{5}{36}}\big(1+13q+73q^2 +338q^3+ 1251q^4+4048q^5+ 11838q^6+\cdots\big)\\
 q^{-\tfrac{1}{12}}\big(1+17q+116q^2+496q^3+1817q^4 +5742q^5+16535q^6+\cdots\big)
 \end{matrix}\right),
\end{equation}
which has the $S$-matrix:
\[
 \rho(S) = \frac{1}{3}\left(\begin{matrix}
2 \cos\big(\frac{5}{18} \pi\big) & -2 \cos\big(\frac{7}{18} \pi\big) & -2 \cos\big(\frac{1}{18} \pi\big) & \sqrt{3} \vspace{1mm}\\
-2 \cos\big(\frac{7}{18} \pi\big) & -2 \cos\big(\frac{1}{18} \pi\big) & -2 \cos\big(\frac{5}{18} \pi\big) & -\sqrt{3} \vspace{1mm}\\
-2 \cos\big(\frac{1}{18} \pi\big) & -2 \cos\big(\frac{5}{18} \pi\big) & 2 \cos\big(\frac{7}{18} \pi\big) & \sqrt{3} \vspace{1mm}\\
\sqrt{3} & -\sqrt{3} & \sqrt{3} & 0
 \end{matrix}\right).
 \]
This representation $\rho$ is realized by a modular tensor category, but we do not know if the conformal modular form $H$ is realized by a strongly regular VOA.

\subsubsection{General numerical examples in rank 4}
The moduli space of irreducible representations of $\Gamma$ of rank $4$ has $3$-dimensional components, giving rise to $3$-parameter families of vector-valued modular forms. It is difficult to handle these $3$-parameter families systematically as above, where we exploited the special nature of the induced representations to solve the differential equations of interest. Nevertheless, one can proceed to search for conformal modular forms arising in the $3$-parameter families using a~numerical approach. Or, in another direction, one can impose more restrictions such as in \cite{ArikeNagatomoSakai}, to make such computations more approachable.

We will describe some numerical experiments that we performed in the cyclic case of \cite{FM5}, and we will assume that $k_1 = 0$, so that $\Tr(L) = 1$. Then, as in ranks $2$ and $3$, up to scaling there is a unique minimal weight form $F$ that satisfies a monic modular linear differential equation of the form
\[
 D^4F + aE_4D^2F + bE_6DF + cE_4^2F = 0
\]
for some $a,b,c \in \CC$. As shown in \cite{FM2}, this corresponds to the ordinary differential equation
\begin{gather*}
\left(\theta_K^{4} - \left(\frac{2K+1}{1-K}\right) \theta_K^{3} + \left(\frac{44 K^{2} - 4(9 a+7)K + 36 a + 11}{36 (1-K)^2}\right) \theta_K^{2} \right. \\
 \left.\qquad{} +\left(\frac{8 K^{2} - 4(3 a + 9 b + 1)K - 6 a + 36 b - 1}{36 (1-K)^2}\right) \theta_K + \frac{c}{(1-K)^2}\right)f = 0.
\end{gather*}
This equation was already considered in Example~\ref{ex:rank4} on page~\pageref{ex:rank4} above, where in particular one finds formulas for the unknowns $a$, $b$ and $c$ in terms of the exponents $e_1$, $e_2$, $e_3$ and $e_4$, which must satisfy $e_1+e_2+e_3+e_4=\Tr(L)=1$ thanks to our hypothesis $k_1=0$.

We have searched through the parameter space, focusing on exponents $e_j$ corresponding to congruence representations, for conformal specializations. We succeeded in finding a number of known examples of characters of VOAs, in addition to a number of as-of-yet undetermined conformal modular forms. The $S$-matrices in these undetermined cases are:
\[
S_1 = \frac{1}{4}\sqrt{1+\frac{1}{\sqrt{5}}}\left(\begin{matrix}
 \sqrt{5}-1&\sqrt{5}-1&2&2\\
 \sqrt{5}-1&1-\sqrt{5}&2&-2\\
 2&2&1-\sqrt{5}&1-\sqrt{5}\\
 2&-2&1-\sqrt{5}&\sqrt{5}-1\\
 \end{matrix}
\right)
\]
or
\begin{gather*}
S_2 = \frac{1}{3}\left(\begin{matrix}
 -\cos\big(\tfrac{5\pi }{18}\big)\!+\!\sqrt{3}\sin\big(\tfrac{5\pi}{18}\big)\! & \cos\big(\tfrac{5\pi}{18}\big)\! + \!\sqrt{3}\sin\big(\tfrac{5\pi}{18}\big) & 2\cos\big(\tfrac{5\pi}{18}\big) & \!\sqrt{3}\vspace{1mm}\\
 \cos\big(\tfrac{5\pi}{18}\big) \!+\! \sqrt{3}\sin\big(\tfrac{5\pi}{18}\big)&2\cos\big(\tfrac{5\pi}{18}\big)&\cos\big(\tfrac{5 \pi}{18}\big)-\sqrt{3}\sin\big(\tfrac{5\pi}{18}\big)&\!\!\!-\sqrt{3}\!\vspace{1mm}\\
 2\cos\big(\tfrac{5\pi}{18}\big)&\!\cos\big(\tfrac{5\pi}{18}\big)\!-\!\sqrt{3}\sin\big(\tfrac{5\pi}{18}\big)\!&-\cos\big(\tfrac{5\pi}{18}\big) \!-\! \sqrt{3}\sin\big(\tfrac{5\pi}{18}\big)\!\!& \!\!\sqrt{3}\vspace{1mm}\\
 \sqrt{3}&-\sqrt{3}&\sqrt{3}&0
 \end{matrix}
\right)\!\!.
\end{gather*}
Note that these two $S$-matrices are realized by known modular tensor categories. The Fourier coefficients of the corresponding suspected conformal modular forms are listed in Table~\ref{rank4examples}. The two examples with $\rho(S) = S_1$ look like they could perhaps be realized as tensor products of strongly regular VOAs with two simple modules each, while those with $\rho(S)=S_2$ could be harder to identify. We end with a result that helps narrow down the possibilities of identifying if the first three entries of Table~\ref{rank4examples} are realized by VOAs:

\begin{table}\centering
 \renewcommand{\arraystretch}{1.8}
 \begin{tabular}{|c|c|}
 \hline
 $\rho(S)$ & $F_V$\\
 \hline
 $S_1$& $q^{-\tfrac {33}{40}}\big(1 + 99q + 50787q^{2} + 2794770q^{3} + 70309800q^{4} + 1134528021q^{5} +\cdots\big)$\\
 &$q^{\tfrac {17}{40}}\big(792 + 154088q + 6610824q^{2} + 145807200q^{3} + 2162364600q^{4} +\cdots\big)$\\
 &$q^{\tfrac {23}{40}}\big(3366 + 466752q + 17581212q^{2} + 361184706q^{3} + 5110157492q^{4} +\cdots\big)$\\
 &$q^{\tfrac {33}{40}}\big(14280 + 1252152q + 39126384q^{2} + 721364424q^{3} + 9486909432q^{4} +\cdots\big)$\\
 \hline
 $S_1$&$q^{-\tfrac {37}{40}}\big(1 + 37q + 65527q^{2} + 5306096q^{3} + 174479457q^{4} + 3487679200q^{5} +\cdots\big)$\\
 &$q^{\tfrac {13}{40}}\big(592 + 223184q + 13516544q^{2} + 383202192q^{3} + 6974809024q^{4} +\cdots\big)$\\
 &$q^{\tfrac {27}{40}}\big(11063 + 1716467q + 75169681q^{2} + 1783793680q^{3} + 28874814615q^{4} +\cdots\big)$\\
 &$q^{\tfrac {37}{40}}\big(47840 + 4779216q + 173590384q^{2} + 3687784672q^{3} + 55362274160q^{4} +\cdots\big)$\\
 \hline
 $S_2$&$q^{\tfrac{-29}{36}} \big( 1 + 58q + 29319q^{2} + 1492282q^{3} +
35652194q^{4} + 551508428q^{5} + \cdots\big)$\\
 &$q^{\tfrac{31}{36}} \big( 16588 + 1295459q + 37792162q^{2} +
661694421q^{3} + 8340292294q^{4} +\cdots\big)$\\
 &$q^{\tfrac{19}{36}} \big( 1595 + 230318q + 8596093q^{2} +
173614474q^{3} + 2409567457q^{4} + \cdots\big)$\\
 &$q^{\tfrac{5}{12}} \big( 1044 + 195489q + 8038422q^{2} +
171114471q^{3} + 2458828278q^{4} + \cdots\big)$\\
 \hline
 $S_2$& $q^{\tfrac{-29}{36}} \big( 1 + 638q + 33959q^{2} + 1509682q^{3} +
35709150q^{4} + 551665608q^{5} + \cdots\big)$\\
 &$q^{\tfrac{-5}{36}} \big( 116 + 18328q + 1302999q^{2} +
37817682q^{3} + 661768081q^{4} +\cdots\big)$\\
 &$q^{\tfrac{19}{36}} \big( 1015 + 228114q + 8586233q^{2} +
173587214q^{3} + 2409491477q^{4} +\cdots\big)$\\
 &$q^{\tfrac{17}{12}} \big( 190269 + 8017542q + 171051831q^{2} +
2458656018q^{3} + 26971011288q^{4} + \cdots\big)$\\
 \hline
\end{tabular}
\caption{Suspected conformal modular forms of rank $4$. The exponents $e_j$ can be read off from the exponents of the leading $q$-powers.}\label{rank4examples}
\end{table}

\begin{lem} The first three examples in Table $\ref{rank4examples}$ are not a tensor product of a pair of VOAs, one of which is holomorphic.
\end{lem}
\begin{proof} Let $V$ be one of these three examples and assume by way of contradiction that $V=X\otimes U$, where $U$, $X$ are VOAs and $X$ is holomorphic. Let $L_V$ be the exponent matrix for $V$, for example, for the first entry of Table~\ref{rank4examples} we have $L_V=\diagg(-33/40, 17/40, 23/40, 33/40)$.

If $X$ has central charge $c_X$ (a positive integer divisible by $8$) then the exponent matrix for $U$ is
\[
L_U= \diagg(c_X/24, c_X/24, c_X/24, c_X/24)+L_V.
\]

The general result that for a strongly regular VOA we have $\tilde{c}>0$ means that the exponent matrix for such a VOA \emph{cannot} be nonnegative in the sense that all entries cannot be nonnegative. Therefore, looking at the first entry of $L_U$, we must have $c_X<24$.

Holomorphic VOAs with central charge $< 24$ are known: they are lattice theories $V_{E_8}$, $V_{E_8\perp E_8}$, $V_{\Gamma_{16}}$. In particular we have $\dim X_1= 248$ or $496$. Therefore $\dim V_1=\dim X_1+\dim U_1\geq 248$. But in the three Examples we have $\dim V_1=99$, $37$ or $58$. This contradiction proves the lemma.
\end{proof}

\subsection*{Acknowledgements}
Franc was supported by an NSERC Discovery Grant, and Mason was supported by grant \#427007 from the Simons Foundation. We thank these institutions for their support. We also thank the anonymous referees for their helpful comments on an earlier draft of this paper.

\pdfbookmark[1]{References}{ref}
\LastPageEnding

\end{document}